\documentclass[a4paper]{amsart}

\usepackage{amsmath,amsthm,amssymb,enumerate}
\usepackage{epsfig}
\usepackage{amssymb}
\usepackage{amsmath}
\usepackage{amssymb}
\usepackage{amsmath,amsthm}
\usepackage[latin1]{inputenc}
\usepackage[T1]{fontenc}
\usepackage{path}
\usepackage{ae,aecompl}
\usepackage{amsfonts}
\usepackage{amsxtra}
\usepackage{euscript,mathrsfs}
\usepackage{color}
\usepackage[left=3.4cm,right=3.4cm,top=4cm,bottom=4cm]{geometry}
\usepackage[citecolor=blue,colorlinks=true]{hyperref}
\allowdisplaybreaks

\usepackage{enumitem}
\setenumerate{label={\rm (\alph{*})}}

\usepackage{amsfonts}
\usepackage{amsxtra}

\usepackage[]{algorithm2e}
\usepackage{framed}
\usepackage{booktabs}
\usepackage{wrapfig}
\usepackage{caption}
\usepackage{subcaption}
\usepackage{tikz}
\usepackage{xcolor}
\definecolor{plot_color1}{HTML}{d7191c}
\definecolor{plot_color2}{HTML}{fdae61}
\definecolor{plot_color3}{HTML}{c2a5cf}
\definecolor{plot_color4}{HTML}{abd9e9}
\definecolor{plot_color5}{HTML}{2c7bb6}

\definecolor{intro_color1}{HTML}{1b9e77}
\definecolor{intro_color2}{HTML}{d95f02}
\definecolor{intro_color3}{HTML}{7570b3}
\definecolor{intro_color4}{HTML}{000000}

\usepackage[frak,hyperref,theorem_section]{paper_diening}
\numberwithin{equation}{section}

\newcommand{\Div}{\divergence}

\newcommand{\R}{\mathbb R}
\newcommand{\N}{\mathbb N}

\newcommand{\E}{\mathbb E}
\newcommand{\p}{\mathbb P}

\newcommand{\F}{\mathfrak F}

\newcommand{\D}{\mathrm d}
\newcommand{\dd}{\mathrm d}
\newcommand{\dx}{\, \mathrm{d}x}

\newcommand{\ds}{\, \mathrm{d}\sigma}
\newcommand{\dt}{\, \mathrm{d}t}

\newcommand{\dxs}{\,\mathrm{d}x\, \mathrm{d}\sigma}
\newcommand{\dif}{\mathrm{d}}

\newcommand{\mf}{\mathfrak{F}}

\newcommand{\prst}{\mathbb{P}}

\newcommand{\mt}{\mathbb{T}^2}
\newcommand{\tor}{\mathbb{T}^2}

\DeclareMathOperator{\diver}{div}

\newcommand{\grSrc}[1]{#1}

\begin{document}

\title[Navier--Stokes equations with transport noise]{Mean Square Temporal error estimates for the 2D stochastic Navier--Stokes equations with transport noise}

\author{D. Breit, T.C. Moyo}
\address{Mathematical Institute, TU Clausthal, Erzstra\ss e 1, 38678 Clausthal-Zellerfeld, Germany}
\email{dominic.breit@tu-clausthal.de, thamsanqa.castern.moyo@tu-clausthal.de}

\author{A. Prohl}
\address{Mathematisches Institut,
Universit\"at T\"ubingen,
Auf der Morgenstelle 10,
D-72076 T\"ubingen
Germany}
\email{prohl@na.uni-tuebingen.de}

\author{J. Wichmann}
\address{School of Mathematics, Monash University, Australia}
\email{joern.wichmann@monash.edu}
\thanks{JW was partially supported by the Australian Government through the Australian Research Council's Discovery
Projects funding scheme (grant number DP220100937).}

\begin{abstract}
We study the 2D Navier--Stokes equation with transport noise subject to periodic boundary conditions.
Our main result is an error estimate for the time-discretisation showing a convergence rate of order (up to) 1/2. It holds with respect to mean square error convergence, whereas previously such a rate for the stochastic Navier--Stokes equations was only known with respect to convergence in probability. 
Our result is based on uniform-in-probability estimates for the continuous as well as the time-discrete solution exploiting the particular structure of the noise.

Eventually, we perform numerical simulations for the corresponding problem on bounded domains with no-slip boundary conditions. They suggest the same convergence rate as proved for the periodic problem hinging sensitively on the compatibility of the data. We also compare the energy profiles with those for corresponding problems with additive or multiplicative It\^o-type noise.
\end{abstract}

\subjclass[2020]{65N15, 60H15, 76D05, 35R60}
\keywords{Stochastic Navier--Stokes equations, transport noise, time discretization, convergence rates, mean square error}

\date{\today}

\maketitle

%
%
%
%
%
%
%
%
%
%

\section{Introduction}

\textbf{Model and motivation.}
In this paper we study the 2D incompressible Navier-Stokes equations with a suitable multiplicative noise of transport type on the torus $\mt\subset\R^2$. Let $Q:=(0,T)\times\mt$ with $T>0$, the system describes the time evolution of a homogeneous fluid where the unknowns are the velocity field $\bfu: (0,T)\times\mt \to \R^2$ and the pressure $\pi:(0,T)\times\mt\to \R$, and it reads

\begin{align}\label{eq:SNS}
\left\{\begin{array}{rc}
\dd\bfu=\mu\Delta\bfu\dt-(\bfu\cdot\nabla)\bfu\dt-\nabla \pi\dt+\sum_{k=1}^K(\bfsigma_k\cdot\nabla)\bfu\circ\dd W_k
& \mbox{in $ Q$,}\\
\Div \bfu=0\qquad\qquad\qquad\qquad\qquad\,\,\,\,& \mbox{in $Q$,}\\
\bfu(0)=\bfu_0\,\qquad\qquad\qquad\qquad\qquad&\mbox{ \,in $\mt$,}\end{array}\right.
\end{align}
on a filtered probability space $(\Omega,\mathfrak F,(\mathfrak F_t),\p)$. The quantity $\mu>0$ denotes the fluid viscosity and $\bfu_0$ is a given initial datum. Here $(W_k)$ are independent real-valued Wiener processes and $\bfsigma_k\in\R^2$  are constant vectors. 


The motivation for transport noise is twofold:
\begin{itemize}
\item In turbulence theory one
is often confronted with resolved large-scale, slow-varying, unresolved small-scale and  fast varying components of the velocity field. With the aim of modelling these effects Holm derived in \cite{HOLM2,HOLM,HOLM1} models with transport noise in fluid dynamics from a physical perspective.
\item Transport noise can have regularisation effects as demonstrated in \cite{FGP} for the transport equation and very recently in \cite{FL} for the 3D incompressible Navier--Stokes equations. In fact, it is shown in \cite{FL} that the {\it blow up} of the solution can be delayed.
Unfortunately, it does not seem possible to restore uniqueness, cf. \cite{HLP,Pa}.
\end{itemize}
Rigorous derivations of stochastic Navier--Stokes equations with transport noise can be found in \cite{DP,FP1,FP2}.

\textbf{Stochastic forcing.} Most of the literature associated with the stochastic Navier--Stokes equations is concerned with stochastic forcing in the sense of It\^{o}, where the noise can be additive or multiplicative. In the 2D case classical results are given in \cite{Ca} and \cite{CC}, where the existence of a unique pathwise solution is shown. Also the spatial regularity of solutions is well-known and can be proved by deterministic estimates, at least in the case of periodic boundary conditions, see; \emph{e.g.}, \cite{KukShi}. By now many deterministic results on the analysis of Navier--Stokes equations have found their stochastic counterparts, and the literature on the numerical approximation is also growing. First convergence rates for the temporal (and spatio-temporal) approximation of the 2D stochastic Navier--Stokes equations were proved in \cite{CP}. In particular, it was shown
that for any $\xi>0$ and $\alpha<1/4$
\begin{align}\label{eq:perror}
&\mathbb P\bigg[\max_{1\leq m\leq M}\|\bfu(t_m)-\bfu_{m}\|_{L^2_x}^2+\sum_{m=1}^M \Delta t\|\nabla\bfu(t_m)-\nabla\bfu_{m}\|_{L^2_x}^2>\xi\,(\Delta t)^{2\alpha}\bigg]\rightarrow0
\end{align}
as $\Delta t\rightarrow0$. This result was eventually improved in \cite{BrDo} to all $\alpha<1/2$, which seems optimal, given the low regularity of the driving Wiener process. In \eqref{eq:perror} we denote by $\bfu_m$ the solution to the time-discrete problem with discretisation parameter $\Delta t=T/M$, and $t_m=m\Delta t$ for some $M\in\N$ and $m=1,\dots,M$. In this scheme the convective term is treated fully implicitly, whereas the multiplicative It\^{o}-type noise is treated explicitly. Note that \eqref{eq:perror} is an estimate for the error with respect to convergence in probability. It seems an intrinsic feature of SPDEs with non-Lipschitz nonlinearity such as \eqref{eq:SNS} that this is the correct error measure (for first results in this direction see \cite{Pr}) and that a similar result with respect to mean square convergence is in general out of reach. However, a convergence rate in mean square for the 2D stochastic Navier--Stokes equations has been established in \cite{BeMi1} with a logarithmic convergence rate, and in \cite{BeMi2} under an assumption concerning the ratio between the viscosity and the strength of the noise.

\textbf{Transport noise.} From a mathematical point of view the most evident advantage between stochastic forcing and transport noise as in \eqref{eq:SNS} (with solenoidal vector fields $\bfsigma_k$) is that the latter conserves energy of each trajectory whereas the former pushes constantly energy into the system. Under the assumption of constant vector fields the noise does not even effect estimates in higher order Sobolev norms. Hence one easily obtains pathwise estimates; see \eqref{eq:enl} below. As far as the time-discretisation is concerned, it is important that a corresponding algorithm preserves these properties of the continuous solution.
With this in mind (and with $\bfu_0$ given)
we seek $\bfu_{m+1}$ as the solution to
\begin{align}\label{tdiscrA}
\begin{aligned}
\bfu_{m+1}&=\bfu_m +\Delta t\Big(-(\bfu_{m+1/2}\cdot\nabla)\bfu_{m+1/2}+\mu\Delta\bfu_{m+1}\Big)\\&\quad+\sum_{k=1}^K(\bfsigma_k\cdot\nabla)\bfu_{m+1/2}\,\Delta_mW_k,\quad\text{where}\,\,\, \bfu_{m+1/2}:=\tfrac{1}{2}(\bfu_{m+1}+\bfu_m),
\end{aligned}
\end{align}
in $W^{1,2}_{\Div}(\mt)'$, 
where $\Delta_m W=W(t_{m+1})-W(t_{m})$. Note that the noise is treated implicitly which makes the error analysis much more complicated compared to the It\^{o}-case. 
We obtain a discrete counterpart of the uniform-in-probability (higher order) energy estimates, see  \eqref{lem:3.1b} below. Though \eqref{eq:enl} and \eqref{lem:3.1b} are just simple observations exploiting the particular structure of the noise, they have striking consequences: 
we are able to prove the error estimate
\begin{align}\label{eq:thm:4intro}
\begin{aligned}
\E\bigg[\max_{1\leq m\leq M}\|\bfu(t_m)-\bfu_{m}\|_{L^2_x}^2&+\sum_{m=1}^M \Delta t \|\nabla\bfu(t_m)-\nabla\bfu_{m}\|_{L^2_x}^2\bigg)\bigg]\leq \,C\,(\Delta t)^{2\alpha}
\end{aligned}
\end{align}
for any $\alpha<1/2$; see Theorem \ref{thm:3.1} for the precise statement. Estimate \eqref{eq:thm:4intro} yields a convergence rate of order (up to) 1/2 for the mean square error. As explained above in \eqref{eq:perror}, previously such a rate was only known for the error with respect to convergence in probability. In fact, one had to exclude large values of $\bfu$ and $\bfu_m$ (in a certain norm) by either neglecting certain sample sets of small probability \cite{BrDo,CP}, working with stopping times \cite{BrPr1,BrPr2} or truncating the nonlinearity \cite{Pr}. Thanks to our pathwise estimates these large values cannot occur as long as the initial datum is sufficiently regular.

\begin{figure}
\includegraphics[width=0.99\textwidth]{\grSrc{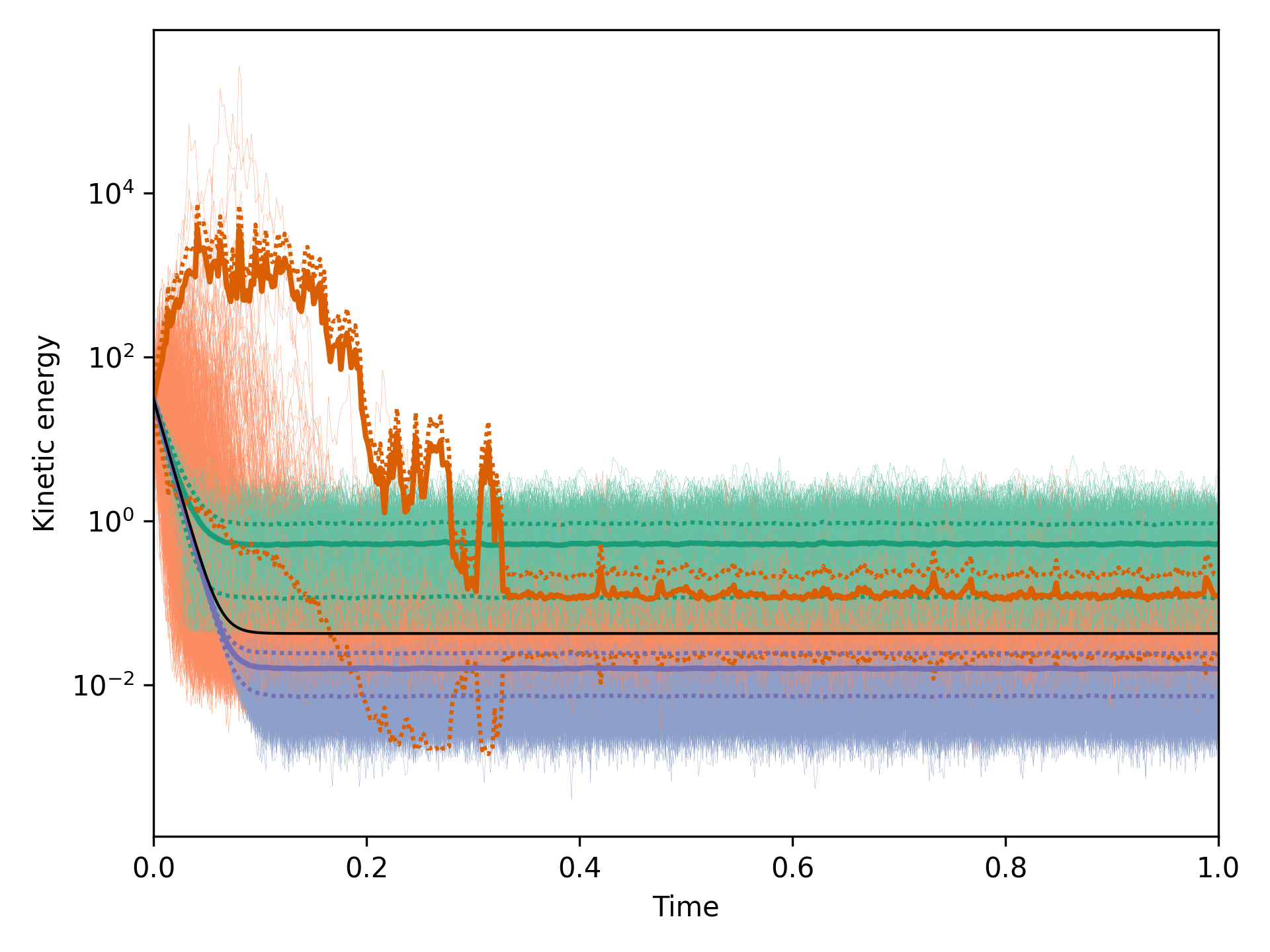}}
\caption{Graphical illustration of the results of SOE--1: time evolution of the  kinetic energy with additive~({\protect\tikz \protect\draw[color=intro_color1, line width=2] (0,0) -- (0.5,0);}), multiplicative~({\protect\tikz \protect\draw[color=intro_color2, line width=2] (0,0) -- (0.5,0);}), transport~({\protect\tikz \protect\draw[color=intro_color3, line width=2] (0,0) -- (0.5,0);}) and no~({\protect\tikz \protect\draw[color=intro_color4, line width=2] (0,0) -- (0.5,0);}) noise. Thick lines and dotted lines show the mean energy and the mean energy plus or minus one standard deviation, respectively. The first 1,000 (out of 10,000) energy trajectories are shown in pale colours. Details on the numerical simulation are given in Section~\ref{sec:numerical-experiments}. }
\label{fig:kinetic}
\end{figure}

\begin{figure}
\includegraphics[width=0.99\textwidth]{\grSrc{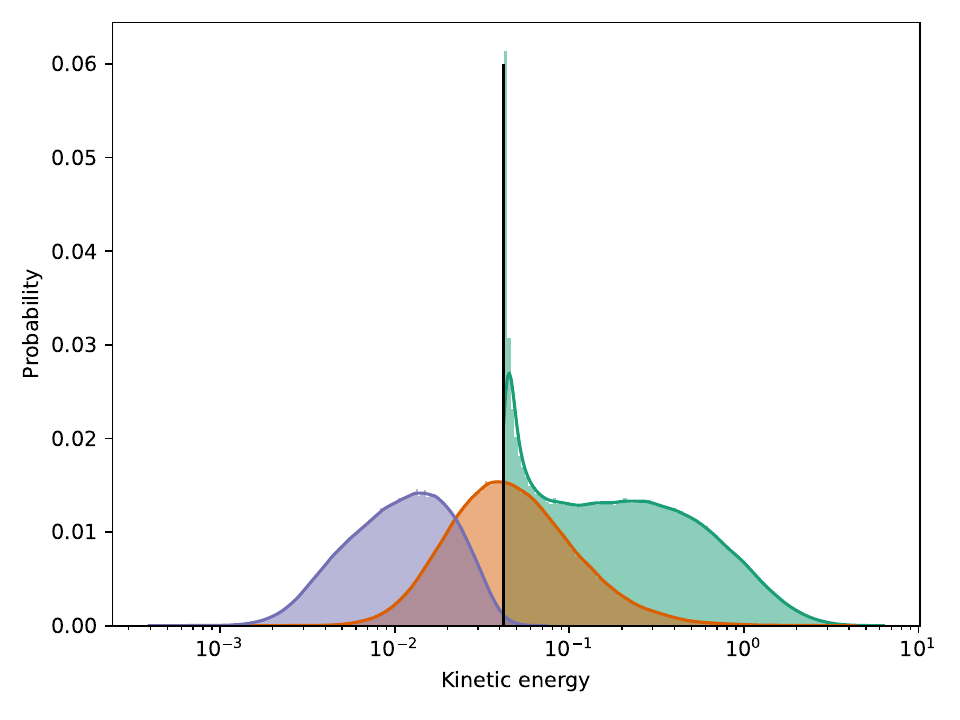}}
\caption{Graphical illustration of the results of SOE--1: empirical approximation (based on 1,000 trajectories) of the stationary distributions of the kinetic energy for additive~({\protect\tikz \protect\draw[color=intro_color1, line width=2] (0,0) -- (0.5,0);}), multiplicative~({\protect\tikz \protect\draw[color=intro_color2, line width=2] (0,0) -- (0.5,0);}) and transport~({\protect\tikz \protect\draw[color=intro_color3, line width=2] (0,0) -- (0.5,0);}) noise. The deterministic stationary energy level is indicated by a black~({\protect\tikz \protect\draw[color=intro_color4, line width=2] (0,0) -- (0.5,0);}) vertical line. Details on the numerical simulation are given in Section~\ref{sec:numerical-experiments}. }
\label{fig:stationary-all}
\end{figure}

\textbf{Numerical simulations.} In Section \ref{sec:numerical-experiments} we perform numerical simulations for the corresponding problem on bounded domains with no-slip boundary conditions and spatially coloured noise (\emph{i.e.}, the driving vector fields $\bfsigma_k$ depend on space). We implement a finite-element based space-time-discretisation in the spirit of
\eqref{tdiscrA}. Again, we must solve a nonlinear problem in each time step to ensure stability of the algorithm.
A key question is if one can expect the same convergence result as for the periodic problem although our theoretical analysis does not apply anymore. Our numerical experiments provide an affirmative answer to this question as long as the data (initial condition and driving vector fields) comply with the constraints, incompressibility and vanishing trace. In fact, our simulations even show that these data conditions are sharp. If initial velocity and noise datum satisfy the constraints, then velocity approximations converge with the theoretically predicted rate from \eqref{eq:thm:4intro}. Conversely, if either initial velocity or noise datum does not satisfy the constraints, then this has a decreasing effect on the convergence rate or it even destroys the convergence. 

Additionally, we investigate the time evolution of the kinetic energy and compare the results for transport noise with those of the same problem with additive noise and multiplicative noise, see Figure~\ref{fig:kinetic}. Different to the latter two cases
transport noise is energy-conservative and thus the kinetic energy is significantly smaller.
In fact,
transport noise even leads to a lower mean energy level compared to the deterministic kinetic energy.
 Initially, dissipation is dominant for transport noise as can be seen by the absence of fluctuations and small standard deviations. At the same time as the deterministic solution reaches its stationary state, randomness becomes more dominant.
The same observation follows from a comparison of the  densities of the stationary distributions, see Figure~\ref{fig:stationary-all}. Additive noise exclusively leads to high energy levels: The density of the stationary distribution attains its maximum at the deterministic stationary energy level, which seems to be a lower bound for each trajectory's energy level. Even if they reach the energy level at an earlier time, they cannot surpass it. Multiplicative noise leads to fairly centred energy levels with occasional exceptions; and the energy of transport noise is mostly concentrated below the deterministic one.

\section{Mathematical framework}
\label{sec:framework}

\subsection{Tensor calculus}

We limit ourselves to 2D space. The system \eqref{eq:SNS} is formulated using standard notations in continuum mechanics. To make the formulation of the problem accessible to all readers we introduce the gradient of a vector field. The gradient of a vector field $\bfu $ is defined to be the second-order tensor
\[
\nabla \bfu = \frac{\partial \bfu}{\partial x_j}\otimes \mathbf{e}_j = \frac{\partial u_i}{\partial x_j}\mathbf{e}_i\otimes \mathbf{e}_j.
\]
In matrix notation,
\begin{equation}\label{eq:grad}
\nabla \bfu=\begin{bmatrix}
\frac{\partial u_1}{\partial x_1}& \frac{\partial u_1}{\partial x_2}\\
\frac{\partial u_2}{\partial x_1}& \frac{\partial u_2}{\partial x_2}
\end{bmatrix}.
\end{equation}
Using \eqref{eq:grad} the equivalence relation $(\nabla \bfu)\bfu =(\bfu\cdot\nabla)\bfu$ is obvious. To distinguish between the products of vectors and matrices we adopt  the following notation. Let $\bfu \in \R^2$ and $\bfv \in \R^2$  be two vectors, then  $\bfu \cdot \bfv$ denotes the product of vectors. Next, for $ \mathbf{A}, \mathbf{B} \in \R^{n\times n}$ we use the notation  $ \mathbf{A}:\mathbf{B} = \mathbf{B}:\mathbf{A} $ to denote $\mathbf{trace}( \mathbf{A}\mathbf{B}^\top )$, that is, the product of matrices.   Let $\mathbf{T}$ be a tensor field, the divergence of a tensor is defined to be the vector
\begin{align}
\mathrm{div}\mathbf{T }= \nabla \mathbf{T}: \mathbb{I}= \frac{\partial \mathbf{T}}{\partial x_i}\mathbf{e}_i = \frac{\partial(T_{jk}\mathbf{e}_j\otimes \mathbf{e}_k)}{\partial x_i}\mathbf{e}_i =\frac{\partial T_{ij}}{\partial x_j}\mathbf{e}_i
\end{align}

For the study of our problem we  need the following identity for tensors
\[
\mathrm{div}(\bfu \otimes \bfv)= (\nabla{\bfu})\bfv+ (\mathrm{div} \,\bfv)\bfu.
\]

\subsection{Function spaces}
All functions spaces are defined over the two-dimensional torus $\mt$ with respect to periodic boundary conditions and zero mean, \emph{i.e.} $\int_{\mt}f\dx=0$. We do not distinguish between scalar- and vector-valued functions. However, vector-valued functions will usually be denoted in bold case.
We denote  by $L^p(\mt)$ and $W^{k,p}(\mt)$ for $p\in[1,\infty]$ and $k\in\mathbb N$, the usual Lebesgue and Sobolev spaces over $\mt$. 
 We consider the subspace
$W^{1,p}_{\Div}(\mt)$ of divergence-free vector fields which is defined accordingly. The space $L^p_{\Div}(\mt)$ is defined as the closure of the set of smooth solenoidal functions in $L^p(\mt).$ We will use the shorthand notations $L^p_x$ and $W^{k,p}_x$ for $L^p(\mt)$ and $W^{k,p}(\mt)$.
For any pair of separable Banach spaces $(X,\|\cdot\|_X)$ and $(Y,\|\cdot\|_Y)$ with $X\subset Y$, we write $X\hookrightarrow Y$ if $X$ is continuously embedded in $Y$, that is $\Vert\cdot\Vert_Y\leq \,C\Vert \cdot\Vert_X$. Here and in what follows, for simplicity, we shall use a generic positive constant $C$ when computing estimates. It can change its value from line to line.

For a separable Banach space $(X,\|\cdot\|_X)$, we denote by $L^p(I;X)$, the set of (Bochner-) measurable functions $u:I\rightarrow X$ such that the mapping $t\mapsto \|u(t)\|_{X}$ belongs to $L^p(I)$. 
The set $C(\overline{I};X)$ denotes the space of functions $u:\overline{I}\rightarrow X$ which are continuous with respect to the norm topology on $(X,\|\cdot\|_X)$. For $\alpha\in(0,1]$ we write
$C^{0,\alpha}(\overline{I};X)$ for the space of H\"older-continuous functions with values in $X$. 
Similarly, for a probability space $(\Omega,\mathfrak F,\p)$ and a separable Banach space $(X,\|\cdot\|_X)$ and $p\in[1,\infty]$
we write $L^p(\Omega,\mathfrak F,\p;X)$ or short
$L^p(\Omega;X)$ for the set of (Bochner-) measurable functions $v:\Omega\rightarrow X$ such that the mapping $\omega\mapsto \|v(\omega)\|_{X}$ belongs to $L^p(\Omega,\mathfrak F,\p)$.

\subsection{Probability setup}

Let $(\Omega,\F,(\F_t)_{t\geq0},\prst)$ be a stochastic basis with a complete, right-continuous filtration and let $(W_k)_{k=1}^K$ be mutually independent real-valued standard Wiener processes relative to $(\F_t)$. We consider smooth solenoidal vector fields $\bfsigma_k:\mt\rightarrow\R^2$.
If $\bfu\in L^2(\Omega;L^2(0,T;W^{1,2}(\mt)))$ is $(\mathfrak F_t)$-adapted the stochastic integrals
\begin{align*}
\int_0^t(\bfsigma_k\cdot\nabla)\bfu \,\D W_k,\quad k=1,\dots, K
\end{align*}
are well-defined in the sense of It\^{o} with values in $L^2(\mt)$. If we only have
$\bfu\in L^2(\Omega;L^2(0,T;L^{2}(\mt)))$ one can use the equality $(\bfsigma_k\cdot\nabla)\bfu=\Div(\bfsigma_k\otimes\bfu)$ and define the stochastic integrals
\begin{align*}
\int_0^t\Div(\bfu\otimes\bfsigma_k) \,\D W_k,\quad k=1,\dots, K
\end{align*}
with values in $W^{-1,2}(\mt)$. 
We define the Stratonovich integrals in \eqref{eq:SNS} by means of the It\^{o}-Stratonovich correction, that is
\begin{align*}
\int_0^t\int_{\mt}\bfu\otimes \bfsigma_k :\nabla\bfphi\dx \,\circ\D W_k&=\int_0^t\int_{\mt}\bfu\otimes \bfsigma_k :\nabla\bfphi\dx \,\D W_k\\&+\frac{1}{2}\Big\langle\Big\langle\int_{\mt}\bfu\otimes \bfsigma_k :\nabla_x\bfphi\dx,W_k\Big\rangle\Big\rangle_t
\end{align*}
for $\bfphi\in W^{1,2}(\mt)$
Here $\langle\langle\cdot,\cdot\rangle\rangle_t$ denotes the cross variation. We compute now the cross variations by means of \eqref{eq:SNS}. We have
\begin{align*}
\int_{\mt}\bfu\otimes \bfsigma_k &:\nabla\bfphi\dx=-\int_{\mt}\bfu_0\otimes \bfsigma_k :\nabla\bfphi\dx+
\int_0^t\int_{\mt}\bfu\otimes\bfu:\nabla(\bfsigma_k\cdot\nabla\bfphi)\dx\,\dif t\\&-\mu\int_0^t\int_{\mt}\nabla\bfu:\nabla(\bfsigma_k\cdot\nabla\bfphi)\dx\,\dif s-\frac{1}{2}\sum_{\ell=1 }^K\int_0^t\nabla(\bfsigma_k\cdot\nabla\bfphi):(\bfsigma_\ell\otimes\bfsigma_\ell\nabla\bfu)\ds\\&
-\sum_{\ell=1}^K\int_0^t\int_{\mt}\bfu\cdot (\bfsigma_\ell \cdot\nabla(\bfsigma_k\cdot\nabla\bfphi))\dx \,\D W_\ell,
\end{align*}
for all $\bfphi\in W^{2,2}(\mt)$.
Here only the last term contributes to the quadratic variation. Plugging the previous considerations together we set
\begin{align} \label{eq:Strato-and-Ito}
&\int_0^t\Div (\bfu\otimes \bfsigma ) \,\circ\D W_k=\int_0^t\Div (\bfu\otimes \bfsigma_k ) \,\D W_k+\frac{1}{2}\int_0^t\Div(\bfsigma_k\otimes\bfsigma_k\nabla\bfu)\ds,
\end{align}
to be understood in $W^{-1,2}(\mt)$ or $W^{-2,2}(\mt)$, depending on the regularity of $\bfu$.

\subsection{The concept of solutions}
\label{subsec:solution}
In two dimensions, pathwise uniqueness for weak solutions is known;
 we refer the reader for instance to \cite{MiRo}. 
Consequently, we may work with the definition of a weak pathwise solution.

\begin{definition}\label{def:inc2d}
Let $(\Omega,\mf,(\mf_t)_{t\geq0},\prst)$ be a given stochastic basis with a complete right-continuous filtration, $(W_k)_{k=1}^K$ being mutually independent real-valued standard Wiener processes relative to $(\F_t)$, and $\bfsigma_k\in\R^2$, $k=1,\dots,K$. 
 Let $\bfu_0$ be an $\mf_0$-measurable random variable. Then $\bfu$ is called a \emph{weak pathwise solution} \index{incompressible Navier--Stokes system!weak pathwise solution} to \eqref{eq:SNS} with the initial condition $\bfu_0$ provided
\begin{enumerate}
\item the velocity field $\bfu$ is $(\mf_t)$-adapted and
$$\bfu \in C([0,T];L^2_{\diver}(\mt))\cap L^2(0,T; W^{1,2}_{\diver}(\tor))\quad\text{$\p$-a.s.},$$
\item the momentum equation
\begin{align}\label{eq:mmnt}
\begin{aligned}
&\int_{\mt}\bfu(t)\cdot\bfvarphi\dx-\int_{\mt}\bfu_0\cdot\bfvarphi\dx
\\&=\int_0^t\int_{\mt}\bfu\otimes\bfu:\nabla\bfvarphi\dx\,\dif t-\mu\int_0^t\int_{\mt}\nabla\bfu:\nabla\bfvarphi\dx\,\dif s\\
&\quad-\sum_{k=1}^K\int_0^t\int_{\mt}\bfu\otimes \bfsigma_k :\nabla \bfvarphi \,\dx\,\D W_k-\frac{1}{2}\sum_{k=1}^K\int_0^t\int_{\mt}\bfsigma_k\otimes\bfsigma_k\nabla\bfu:\nabla \bfvarphi\dx\,\ds
\end{aligned}
\end{align}
holds $\p$-a.s. for all $\bfvarphi\in C^\infty_{\diver}(\mt)$ and all $t\in[0,T]$.
\end{enumerate}
\end{definition}

\begin{theorem}\label{thm:inc2d}
Let $(\Omega,\mf,(\mf_t)_{t\geq0},\prst)$ be a given stochastic basis with a complete right-continuous filtration, $(W_k)_{k=1}^K$ being mutually independent real-valued standard Wiener processes relative to $(\F_t)$, and $\bfsigma_k\in\R^2$, $k=1,\dots,K$. 
 Let $\bfu_0\in L^2_{\Div}(\mt)$.
Then there exists a unique weak pathwise solution to \eqref{eq:SNS} in the sense of Definition \ref{def:inc2d} with the initial condition $\bfu_0$.
\end{theorem}
Due to the particular structure of the noise, it is energy conservative and one easily obtains $\p$-a.s. (by applying It\^{o}'s formula to $t\mapsto \frac{1}{2}\|\bfu(t)\|_{L^2_x}^2$)
\begin{align}
\label{eq:en0}\tfrac{1}{2}\|\bfu\|^{2}_{L^{2}_x}+\int_0^t\|\nabla\bfu\|^2_{L^2_x}\dt&=\,\tfrac{1}{2}\|\bfu_0\|^{2}_{L^{2}_x}\quad\forall \,\,0\leq t\leq T.
\end{align}
Note that this equality even holds if the $\bfsigma_k$'s are divergence-free vector fields. For constant vectors $\bfsigma_k$ we even obtain $\p$-a.s.
\begin{align}
\label{eq:en1}\tfrac{1}{2}\|\nabla\bfu\|^{2}_{L^{2}_x}+\int_0^t\|\Delta\bfu\|^2_{L^2_x}\dt&=\,\tfrac{1}{2}\|\nabla\bfu_0\|^{2}_{L^{2}_x}\quad\forall \,\,0\leq t\leq T.
\end{align}
Note that the convective term cancels here due to the periodic boundary conditions (but could also be estimated accordingly by Ladyshenskaya's inequality).
Taking into account \eqref{eq:en0} and \eqref{eq:en1} one can iteratively prove higher order estimates: the stochastic integral always cancels for constant vectors $\bfsigma_k$, while the convective terms can be handled by well-known deterministic estimates. Hence we obtain for any $l\in\N_0$
\begin{align}\label{eq:enl}
\sup_{0\leq t\leq T}\tfrac{1}{2}\|\bfu\|^{2}_{W^{l,2}_x}+\int_0^T\|\bfu\|^2_{W^{l+1,2}_x}\dt&\leq \,C
\end{align}
 $\p$-a.s. with a deterministic constant $C$ depending on $\|\bfu_0\|_{W^{l,2}_x}$,
 provided that $\bfu_0\in W^{l,2}(\mt)$.

\subsection{Pressure decomposition}

 For $\bfphi\in C^\infty(\mt)$ we can insert $\bfphi-\nabla\Delta^{-1}\Div\bfphi$
and obtain
\begin{align}
\nonumber
\int_{\mt}\bfu(t)\cdot\bfvarphi\dx &+\int_0^t\int_{\mt}\mu\nabla\bfu:\nabla\bfphi\dxs-\int_0^t\int_{\mt}\bfu\otimes\bfu:\nabla\bfphi\dxs\\
\label{eq:pressurecon}&=\int_{\mt}\bfu(0)\cdot\bfvarphi\dx
+\int_0^t\int_{\mt}\pi_{\mathrm{det}}\,\Div\bfphi\dxs
\\
\nonumber&\quad-\sum_{k=1}^K\int_0^t\int_{\mt}\bfu\otimes \bfsigma_k :\nabla \bfphi \,\dx\,\D W_k-\frac{1}{2}\sum_{k=1}^K\int_0^t\int_{\mt}\bfsigma_k\otimes\bfsigma_k\nabla\bfu:\nabla\bfphi\,\dx\,\ds\\
\nonumber&\quad+\sum_{k=1}^K\int_0^t\int_{\mt}\Pi_k^{\mathrm{Ito}}\cdot \bfphi \,\dx\,\D W_k+\frac{1}{2}\sum_{k=1}^K\int_0^t\int_{\mt}\Pi_k^{\mathrm{cor}}\cdot\bfphi\,\dx\,\ds,
\end{align}
where 
\begin{align*}
\pi_{\mathrm{det}}&=-\Delta^{-1}\Div\Div\big(\bfu\otimes\bfu\big),\\
\Pi_k^{\mathrm{Ito}}&=-\nabla\Delta^{-1}\Div\Div (\bfu\otimes \bfsigma_k ),\\
\Pi_k^{\mathrm{cor}}&=-\nabla\Delta^{-1}\Div\Div(\bfsigma_k\otimes\bfsigma_k\nabla\bfu).
\end{align*}
This corresponds to the stochastic pressure decomposition from \cite{BrDo} (see also \cite{Br} and \cite[Chap. 3]{Br2}). However, we have additional terms due to the It\^{o}-Stratonovich correction. 

If $\bfu_0\in W^{1,2}_{\Div}(\mt)$, \eqref{eq:en1} yields that equation \eqref{eq:SNS} is satisfied strongly in the analytical sense. That is
we have
\begin{align}\label{eq:strong}
\begin{aligned}
\bfu(t)&=\bfu(0)+\int_0^t\Big[\mu\Delta\bfu-(\nabla\bfu)\bfu
-\nabla\pi_{\mathrm{det}}\Big]\ds\\
&\quad+\sum_{k=1}^K\int_0^t\Div (\bfu\otimes \bfsigma_k ) \,\D W_k+\frac{1}{2}\sum_{k=1}^K\int_0^t\Div(\bfsigma_k\otimes\bfsigma_k\nabla\bfu)\ds\\
&\quad+\sum_{k=1}^K\int_0^t\Pi_k^{\mathrm{Ito}} \,\D W_k+\frac{1}{2}\sum_{k=1}^K\int_0^t\Pi_k^{\mathrm{cor}}\ds
\end{aligned}
\end{align}
$\p$-a.s., for all $t\in[0,T]$, recall equation \eqref{eq:pressurecon}. As in \cite[Cor. 1]{BrDo} we obatin the following estimates for the ``deterministic pressure''.
\begin{corollary}\label{cor:pressure}
Let $l\in\N_0$ and $\bfu_0\in W^{l,2}_{\Div}(\mt)$ and $\bfsigma_k\in\R^2$ for $k=1,\dots, K$. Then we have $\pi_{\mathrm{det}}\in L^2(0,T;W^{l,2}(\mt))$ $\mathbb P$-a.s. and
\begin{align*}
\int_0^T\|\pi_{\mathrm det}\|_{W^{l,2}_x}^2\dt\leq\,C
\end{align*}
uniformly in $\Omega$, where $C$ depends on $\|\bfu_0\|_{W^{l,2}_x}$.
\end{corollary}
One easily derives the following estimates
from the definitions of $\Pi_k^{\mathrm{Ito}}$ and $\Pi_k^{\mathrm{cor}}$ and the smoothness of the $\bfsigma_k$:
\begin{align*}
\|\Pi_k^{\mathrm{Ito}}\|_{W^{l,2}_x}\leq\,C\|\bfu\|_{W^{l+1,2}_x},\quad \|\Pi_k^{\mathrm{cor}}\|_{W^{l,2}_x}\leq\,C\|\bfu\|_{W^{l+2,2}_x},
\end{align*}
for $l=0,1,2$, where $W^{0,2}(\mt)=L^2(\mt)$.
Hence the stochastic terms in \eqref{eq:strong} can be controlled by
\begin{align*}
\E\bigg[ &\bigg\|\int_0^{\cdot}\Big(\Div (\bfu\otimes \bfsigma_k )+\Pi_k^{\mathrm{Ito}}\Big) \,\D W_k\bigg\|_{C^{\alpha}([0,T];W^{l,2}_x)}^{r}\bigg]\\&\leq\,C\,\E\bigg[ \int_0^T\Big\|\Big(\Div (\bfu\otimes \bfsigma_k )+\Pi_k^{\mathrm{Ito}}\Big) \Big\|^r_{W_x^{l,2}}\dt \bigg]\\&\leq\,C\,\E\bigg[\int_0^T\|\bfu\|^r_{W_x^{l+1,2}}\dt \bigg],\\
\E\bigg[&\bigg\|\int_0^{\cdot}\Big(\Div (\bfsigma_k\otimes \bfsigma_k\nabla\bfu )+\Pi_k^{\mathrm{cor}}\Big) \ds\bigg\|_{C^{\alpha}([0,T];W^{l,2}_x)}^{r}\bigg]\\&\leq\,C\,\E\bigg[\bigg\|\int_0^{\cdot}\Big(\Div (\bfsigma_k\otimes \bfsigma_k\nabla\bfu )+\Pi_k^{\mathrm{cor}}\Big) \ds\bigg\|_{W^{1,2}([0,T];W^{l,2}_x)}^{r}\bigg]\\&\leq\,C\,\E\bigg[\bigg(\int_0^T\|\bfu\|^2_{W_x^{l+2,2}}\dt\bigg)^{\frac{r}{2}}\bigg],
\end{align*}
where $r>2$ and $\alpha\in(0,\frac{1}{2}-\frac{1}{r})$.
Combining this with (\ref{eq:enl}) and Corollary \ref{cor:pressure} we obtain
the following result on 
the time regularity of the velocity field arguing similarly to  \cite[Cor. 2]{BrDo}.
\begin{corollary}\label{cor:uholder}
Suppose that $\bfu_0\in W^{l,2}_{\Div}(\mt)$ for some $l\in \N$ and $\bfsigma_k\in\R^2$ for $k=1,\dots, K$. Then we have
\begin{align}
\label{eq:holder}
\E\Big[\|\bfu\|_{C^\alpha([0,T];W^{l-1,2}_x)}^{\frac{r}{2}} \Big]<\infty,
\end{align}
for all $r > 2$ and  $\alpha<\frac{1}{2} - \frac{1}{r}$.
\end{corollary}

\section{Time-discretization}
\subsection{Preparations}
We consider an equidistant partition of $[0,T]$ with mesh size $\Delta t=T/M$ and set $t_m=m\Delta t$. 
Let $\bfu_0$ be an $\mathfrak F_0$-measurable random variable with values in $W^{1,2}_{\Div}(\mt)$.
We aim at constructing iteratively a sequence of $\mathfrak F_{t_{m+1}}$-measurable random variables $\bfu_{m+1}$ with values in $W^{1,2}_{\Div}(\mt)$ such that
for every $\bfphi\in W^{1,2}_{\Div}(\mt)$ it holds true $\p$-a.s.
\begin{align}\label{tdiscr}
\begin{aligned}
\int_{\mt}&\bfu_{m+1}\cdot\bfvarphi \dx +\Delta t\bigg(\int_{\mt}(\nabla\bfu_{m+1/2})\bfu_{m+1/2}\cdot\bfphi\dx+\mu\int_{\mt}\nabla\bfu_{m+1}:\nabla\bfphi\dx\bigg)\\
&\qquad=\int_{\mt}\bfu_{m}\cdot\bfvarphi \dx+\sum_{k=1}^K\int_{\mt}(\bfsigma_k\cdot\nabla)\bfu_{m+1/2}\,\Delta_mW_k\cdot \bfvarphi\dx,
\end{aligned}
\end{align}
where $\Delta_m W=W(t_{m+1})-W(t_{m})$. The existence of $\bfu_{m+1}$ can be shown via a standard fixed point argument: rewriting \eqref{tdiscr} in the unknown $\bfu_{m+1/2}:=\frac{1}{2}(\bfu_{m+1}+\bfu_m)$, the nonlinearity as well as the noise-term cancel when testing with $ \bfu_{m+1/2}$ such that the classical Browder theorem applies and we obtain
$\bfu_{m+1/2}$ given for $\bfu_m$ and $\Delta_m W$. Finally we set $\bfu_{m+1}:=2\bfu_{m+1/2}- \bfu_m$. Its $\mathfrak F_{t_{m+1}}$-measurability follows from the continuity of the fixed point map.
Testing \eqref{tdiscr} with $\bfu_{m+1/2}$ we observe that both, the convective term and the noise term disappear provided $\Div\bfsigma_k=0$ (here it is  not necessary to take them constant). We obtain
\begin{align*}
\tfrac{1}{2}\|\bfu_{m+1}\|^{2}_{L^{2}_x}+\tfrac{1}{4}\mu\Delta t\|\nabla\bfu_{m+1}\|^{2}_{L^{2}_x} + \mu\Delta t\|\nabla\bfu_{m+1/2}\|^{2}_{L^{2}_x}=\tfrac{1}{2}\|\bfu_m\|^{2}_{L^{2}_x} + \tfrac{1}{4}\mu\Delta t\|\nabla\bfu_{m}\|^{2}_{L^{2}_x},
\end{align*}
where we used the matrix identities 
\begin{align*}
\nabla \bfu_{m+1}: \nabla \bfu_{m+\frac{1}{2}} &= \abs{\nabla \bfu_{m+\frac{1}{2}}}^2 +  \frac{1}{4}\big(\nabla  \bfu_{m+1} - \nabla \bfu_{m} \big) :\big(\nabla  \bfu_{m+1} + \nabla \bfu_{m} \big) \\
&=  \abs{\nabla \bfu_{m+\frac{1}{2}}}^2 +  \frac{1}{4}\big( \abs{\nabla \bfu_{m+1}}^2 - \abs{\nabla \bfu_{m}}^2 \big). 
\end{align*}
After iteration, we find that
\begin{align*}
\tfrac{1}{2}\|\bfu_{m+1}\|^{2}_{L^{2}_x}+\tfrac{1}{4}\mu\Delta t\|\nabla\bfu_{m+1}\|^{2}_{L^{2}_x}+\mu\Delta t\sum_{n=0}^m\|\nabla\bfu_{n+1/2}\|^{2}_{L^{2}_x}=\tfrac{1}{2}\|\bfu_0\|^{2}_{L^{2}_x} +  \tfrac{1}{4}\mu\Delta t\|\nabla\bfu_{0}\|^{2}_{L^{2}_x}.
\end{align*}
In particular, we have
\begin{align}
\label{lem:3.1a0}
\begin{aligned}
\max_{1\leq m\leq M}\tfrac{1}{2}\|\bfu_m\|^{2}_{L^{2}_x}&+ \max_{1\leq m\leq M}\tfrac{1}{4}\mu\Delta t\|\nabla\bfu_{m}\|^{2}_{L^{2}_x} +\mu\Delta t\sum_{m=1}^M\|\nabla\bfu_{m+1/2}\|^2_{L^2_x}\\
&\leq \frac{3}{2} \left( \|\bfu_0\|^{2}_{L^{2}_x} +\tfrac{1}{4}\mu\Delta t\|\nabla\bfu_{0}\|^{2}_{L^{2}_x} \right) .
\end{aligned}
\end{align}
If the $\bfsigma_k$ are constant we can argue similarly for the test with $\Delta\bfu_{m+1/2}$, obtaining
\begin{align}\label{lem:3.1a}
\begin{aligned}
\max_{1\leq m\leq M}\|\nabla\bfu_m\|^{2}_{L^{2}_x} &+ \max_{1\leq m\leq M} \mu\Delta t\|\nabla^2 \bfu_{m}\|^{2}_{L^{2}_x} +\Delta t\mu\sum_{m=1}^M\|\nabla^2 \bfu_{m+1/2}\|^2_{L^2_x}\\
&\leq \frac{3}{2} \left( \|\nabla \bfu_0\|^{2}_{L^{2}_x} +\tfrac{1}{4}\mu\Delta t\|\nabla^2\bfu_{0}\|^{2}_{L^{2}_x} \right).
\end{aligned}
\end{align}
provided $\bfu_0\in W^{2,2}_{\Div}(\mt)$. One can obtain higher order estimates in the same manner: the noise cancels, while the convective term can be estimated (using \eqref{lem:3.1a0} and \eqref{lem:3.1a}) by well-known deterministic estimates. Hence we have for any $l\in\N_0$
\begin{align}
\label{lem:3.1b}\max_{1\leq m\leq M}\|\bfu_m\|^{2}_{W^{l,2}_x}+ \max_{1\leq m\leq M} \mu\Delta t\|\bfu_m\|^{2}_{W^{l+1,2}_x} + \mu\Delta t\sum_{m=1}^M\|\bfu_{m+1/2}\|^2_{W^{l+1,2}_x}&\leq\,C
\end{align}
with $C$ depending on $\|\bfu_0\|_{W^{l+1,2}_x}$, provided that $\bfu_0\in W^{l+1,2}_{\Div}(\mt)$. 

\subsection{Temporal error estimate}
In the following we formulate our main result.
\begin{theorem}\label{thm:3.1}
Assume that $\bfu_0\in W^{4,2}_{\Div}(\mt)$ and $\bfsigma_k\in\R^2$ for $k=1,\dots, K$. Let $\bfu$ and $(\bfu_m)_{m=1}^M$ be the unique strong solution
to \eqref{eq:SNS} in the sense of Definition \ref{def:inc2d} and the solution to \eqref{tdiscr}, respectively. Suppose that we have
\begin{align} \label{ass:vis-dom-noise}
\Gamma_\bfsigma:= \frac{9}{8}\sum_{k=1}^K |\bfsigma_k|^2 < \mu.
\end{align}
Then the following error estimate is satisfied: for any~$\alpha < 1/2$ there exists~$C>0$ (depending on $\alpha$, $T$, $\|\bfu_0\|_{W^{4,2}_x}$, $\mu$, $\Gamma_\bfsigma$, $K$) such that
\begin{align}\label{eq:thm:4}
\begin{aligned}
\E\bigg[\max_{1\leq m\leq M}\|\bfu(t_m)-\bfu_{m}\|^2_{L^2_x}&+\Delta t\sum_{m=1}^M  \|\nabla\bfu(t_m)-\nabla\bfu_{m}\|^2_{L^2_x}\bigg]\leq \,C\,(\Delta t)^{2\alpha}.
\end{aligned}
\end{align} 
\end{theorem}
The proof of this theorem uses the error at integer times $\bfe_{m+1} := \bfu(t_{m+1}) - \bfu_{m+1}$ rather than the arithmetic mean of consecutive errors~$\bfe_{m+1/2}:= \tfrac{1}{2}(\bfe_{m+1} + \bfe_{m})$, which is in contrast to the derivation of the discrete stability that relies on the arithmetic mean (see e.g. Equation~\eqref{lem:3.1a0}); this approach in combination with the proper balancing of noise and diffusion (see Condition~\eqref{ass:vis-dom-noise}) then allows control of the stochastic terms.
\begin{proof}
Let $m \in \{0,\ldots,M-1\}$. Subtracting  \eqref{tdiscr} from \eqref{eq:mmnt} yields the local error equation
\begin{align*}
\begin{aligned}
&\int_{\mt}\bfe_{m+1}\cdot\bfvarphi \dx +\int_{t_{m}}^{t_{m+1}}\bigg(\int_{\mt} \Big((\bfu(t)\cdot\nabla)\bfu(t)-(\bfu_{m+\frac{1}{2}}\cdot\nabla)\bfu_{m+\frac{1}{2}} \Big)\cdot \bfvarphi\dx \bigg) \dt\\
&\qquad \qquad +\mu\int_{t_{m}}^{t_{m+1}}\int_{\mt}\nabla (\bfu(t)-\bfu_{m+1}):\nabla\bfphi\dx\dt\\
&=\int_{\mt}\bfe_{m}\cdot\bfvarphi \dx
+\sum_{k=1}^K\int_{\mt}\left(\int_{t_{m}}^{t_{m+1}}(\bfsigma_k\cdot \nabla )\bfu \circ\,\dd W_k-(\bfsigma_k\cdot\nabla)\bfu_{m+\frac{1}{2}}\Delta_m W_k\right)\cdot \bfvarphi\dx,
\end{aligned}
\end{align*}
which holds true $\p$-a.s. for any $\bfphi\in W^{1,2}_{\Div}(\mt)$. Slightly reorganizing the equation shows
\begin{align*}
&\int_{\mt}\big( \bfe_{m+1} - \bfe_{m} \big) \cdot\bfvarphi \dx +  \mu \Delta t \int_{\mt} \nabla \bfe_{m+1} : \nabla \bfvarphi \dx \\
&\qquad = \int_{\mt} \underbrace{\mu \int_{t_{m}}^{t_{m+1}} \big(\nabla\bfu(t_{m+1})-\nabla\bfu(t)\big) \dt}_{=:I_1(m)} :\nabla \bfvarphi \dx
\\
&\qquad \qquad+\int_{\mt} \underbrace{\int_{t_{m}}^{t_{m+1}} \Big((\bfu_{m+\frac{1}{2}}\cdot\nabla)\bfu_{m+\frac{1}{2}}-(\bfu(t)\cdot\nabla)\bfu(t)\Big) \dt }_{=:I_2(m)} \cdot \, \bfvarphi \dx \\
&\qquad\qquad+\int_{\mt} \underbrace{\sum_{k=1}^K\left(\int_{t_{m}}^{t_{m+1}}(\bfsigma_k\cdot \nabla )\bfu \circ\,\dd W_k-(\bfsigma_k\cdot\nabla )\bfu_{m+\frac{1}{2}}\Delta_m W_k\right) }_{=:I_3(m)} \cdot \, \bfvarphi \dx.
\end{align*}
Next, we choose~$\bfvarphi = \bfe_{m+1} \in W^{1,2}_{\Div}(\mt)$, which implies
\begin{align} \label{eq:local-error-id}
\begin{aligned}
&\int_{\mt}\big( \bfe_{m+1} - \bfe_{m} \big) \cdot \bfe_{m+1} \dx +  \mu \Delta t \int_{\mt} \nabla \bfe_{m+1} : \nabla \bfe_{m+1} \dx \\
&\qquad =  \int_{\mt} I_1(m) : \nabla \bfe_{m+1} \dx + \int_{\mt} I_2(m) \cdot \bfe_{m+1} \dx  + \int_{\mt} I_3(m) \cdot  \bfe_{m+1} \dx.
\end{aligned}
\end{align}
Identity~\eqref{eq:local-error-id} is a time local error identity, which forms the starting point of our error analysis. Before we move on to obtain a time global error estimate (by summing the local error identity for different time instances), we rewrite the first term on the left-hand-side in the desired form of our claimed Inequality~\eqref{eq:thm:4}. To do so, let us recall the following algebraic identity that is valid for all inner-products (we present it here for vectors $\bfa,\bfb\in\R^2$): 
\begin{align*} 
(\bfa-\bfb) \cdot \bfa &=\frac{1}{2}\left(|\bfa|^2-|\bfb|^2+|\bfa-\bfb|^2\right).
\end{align*}
Using this identity for $\bfa = \bfe_{m+1}$ and $\bfb = \bfe_m$, we find that
\begin{align*}
\int_{\mt}\big( \bfe_{m+1} - \bfe_{m} \big) \cdot \bfe_{m+1} \dx = \frac{1}{2}\left( \norm{\bfe_{m+1}}_{L^2_x}^2 -\norm{\bfe_{m}}_{L^2_x}^2 \right) + \frac{1}{2} \norm{\bfe_{m+1} - \bfe_{m}}_{L^2_x}^2.
\end{align*}

We are ready to move on to the time global error analysis: Let $n \in \{1,\ldots, M\}$. We sum the local error equation (Equation~\eqref{eq:local-error-id}) for $m \in \{0, \ldots, n-1\}$. Due to the telescopic structure of the terms on the left-hand-side and using~$\bfe_0 = 0$, we find that
\begin{align*}
&\frac{1}{2} \norm{\bfe_n}_{L^2_x}^2 + \frac{1}{2} \sum_{m=0}^{n-1} \norm{\bfe_{m+1} - \bfe_{m}}_{L^2_x}^2 + \mu \Delta t \sum_{m=0}^{n-1} \norm{\nabla\bfe_{m+1}}_{L^2_x}^2 \\
&\qquad =  \sum_{m=0}^{n-1} \left( \int_{\mt} I_1(m) : \nabla \bfe_{m+1} \dx + \int_{\mt} I_2(m) \cdot \bfe_{m+1} \dx  + \int_{\mt} I_3(m) \cdot  \bfe_{m+1} \dx \right).
\end{align*}
Observe that all terms on the left-hand-side are non-negative. Therefore, by neglecting all but one term, which we call~$\bfX_n$ and can be either of them; taking the maximum over~$n \in \{1,\ldots,M\}$ and applying expectations we arrive at 
\begin{align*}
&\mathbb{E}\left[ \max_{1\leq n\leq M} \bfX_n \right] \leq \mathbb{E}\left[ \max_{1\leq n\leq M} \sum_{m=0}^{n-1} \int_{\mt} I_1(m) : \nabla \bfe_{m+1} \dx  \right] \\
&\qquad + \mathbb{E}\left[ \max_{1\leq n\leq M} \sum_{m=0}^{n-1} \int_{\mt} I_2(m) \cdot \bfe_{m+1} \dx   \right] + \mathbb{E}\left[ \max_{1\leq n\leq M} \sum_{m=0}^{n-1} \int_{\mt} I_3(m) \cdot  \bfe_{m+1} \dx  \right].
\end{align*}
Invoking the above inequality for each choice of $\bfX_n$ yields
\begin{align}\label{eq:error-est-pre-Gronwall}
\begin{aligned}
&\frac{1}{2}  \mathbb{E}\left[ \max_{1\leq n\leq M} \norm{\bfe_n}_{L^2_x}^2\right]+\frac{1}{2} \mathbb{E}\left[ \sum_{m=0}^{M-1} \norm{\bfe_{n+1} - \bfe_{n}}_{L^2_x}^2 \right] +\mu  \mathbb{E}\left[ \Delta t \sum_{m=0}^{M-1} \norm{\nabla\bfe_{n+1}}_{L^2_x}^2 \right]  \\
&\qquad \leq 3 (\mathrm{I} + \mathrm{II} + \mathrm{III}), 
\end{aligned}
\end{align}
where 
\begin{align*}
\mathrm{I} &:= \mathbb{E}\left[ \max_{1\leq n\leq M} \sum_{m=0}^{n-1} \int_{\mt} I_1(m) : \nabla \bfe_{m+1} \dx  \right], \\
\mathrm{II} &:= \mathbb{E}\left[ \max_{1\leq n\leq M} \sum_{m=0}^{n-1} \int_{\mt} I_2(m) \cdot \bfe_{m+1} \dx   \right], \\
\mathrm{III} &:= \mathbb{E}\left[ \max_{1\leq n\leq M} \sum_{m=0}^{n-1} \int_{\mt} I_3(m) \cdot \bfe_{m+1} \dx   \right].
\end{align*}
Inequality~\eqref{eq:thm:4} will follow from an application of Gronwall's lemma once we have established the following estimates: for arbitrary~$\delta > 0$, $\beta > 0$ and $\alpha \in (0,1/2)$ there exist constants~$C$ and $C(\delta)$ (both independent of~$\beta$) such that
\begin{subequations} \label{eq:error-terms-all}
\begin{align} \label{eq:error-term-1}
\mathrm{I} &\leq \delta \mathbb{E}\left[\Delta t \sum_{m=0}^{M-1} \norm{\nabla \bfe_{m+1} }_{L^2_x}^2 \right]  + C(\delta) (\Delta t)^{2\alpha}, \\ \label{eq:error-term-2}
\mathrm{II} &\leq C (\Delta t)^{2\alpha} +  \delta  \mathbb{E}\left[ \Delta t  \sum_{m=0}^{M-1}  \norm{ \nabla \bfe_{m+1} }_{L^2_x}^2  \right] +  C(\delta) \mathbb{E}\left[ \Delta t  \sum_{m=1}^{M} \max_{1\leq n \leq m} \norm{\bfe_{n}}_{L^2_x}^2  \right], \\ \label{eq:error-term-3}
\mathrm{III} &\leq \delta \mathbb{E}\left[  \sum_{m=0}^{M-1} \norm{\bfe_{m+1} - \bfe_{m}}_{L^2_x}^2  \right] + \delta \mathbb{E}\left[ \max_{1\leq n \leq M} \norm{\bfe_{n}}_{L^2_x}^2 \right]  + C(\delta) (\Delta t)^{2\alpha} \\ \nonumber
&\qquad + \frac{\beta}{4}  \mathbb{E}\left[  \sum_{m=0}^{M-1} \norm{\bfe_{m+1} - \bfe_{m}}_{L^2_x}^2  \right]  +  \frac{1}{4 \beta}  \sum_{k=1}^K \abs{\bfsigma_k}^2  \mathbb{E}\left[  \Delta t \sum_{m=0}^{M-1}  \norm{\nabla \bfe_{m+1}}_{L^2_x}^2 \right].
\end{align}
\end{subequations}
For the moment, let us assume these inequalities have already been established (we verify them later). Using them in Inequality~\eqref{eq:error-est-pre-Gronwall} yields
\begin{align} \label{eq:pre-gronwall}
\begin{aligned}
&\frac{1}{2}\left( 1 - 6\delta \right)  \mathbb{E}\left[ \max_{1\leq n\leq M} \norm{\bfe_n}_{L^2_x}^2\right]+\frac{1}{2}\left( 1 - 6\delta - \frac{3}{2}\beta \right)  \mathbb{E}\left[ \sum_{m=0}^{M-1} \norm{\bfe_{m+1} - \bfe_{m}}_{L^2_x}^2 \right] \\
&\qquad \qquad +\left(\mu - 6\delta -  \frac{3}{4\beta}  \sum_{k=1}^K \abs{\bfsigma_k}^2 \right) \mathbb{E}\left[ \Delta t \sum_{m=0}^{M-1} \norm{\nabla\bfe_{m+1}}_{L^2_x}^2 \right]  \\
&\qquad \leq 3 (C + 2C(\delta)) (\Delta t)^{2\alpha}  + 3 C(\delta) \mathbb{E}\left[ \Delta t \sum_{m=1}^{M}  \max_{1\leq n \leq m} \norm{\bfe_{n}}_{L^2_x}^2  \right].
\end{aligned}
\end{align}
We intend to apply Gronwall's lemma to $m \mapsto  \mathbb{E}\left[ \max_{1\leq n\leq m} \norm{\bfe_n}_{L^2_x}^2\right]$. For this, we need to make sure that the coefficients of the other terms are non-negative. We even impose positivity of the coefficients, since this allows us to additionally derive error estimates for the gradient error and the time-shifted error. The conditions on the coefficients read:
\begin{subequations}
\begin{align} \label{eq:1st-coeff}
1 - 6\delta - \frac{3}{2}\beta > 0 \qquad \Leftrightarrow \qquad \beta < \frac{2}{3}- 4 \delta,
\end{align}
and
\begin{align} \label{eq:2nd-coeff}
\mu - 6\delta -  \frac{3}{4\beta}  \sum_{k=1}^K |\bfsigma_k|^2 > 0.
\end{align}
\end{subequations}
Notice that $\delta >0$ is still to our disposal, which allows us to choose~$\beta$ close to~$2/3$. To do this rigorously, let~$\varepsilon >0$. Now, we use the ansatz $\beta := \frac{2}{3} - \varepsilon$ and check the constraints on~$\delta$ that are imposed by the Inequalities~\eqref{eq:1st-coeff} and~\eqref{eq:2nd-coeff}. The first inequality imposes $\delta < \frac{\varepsilon}{4}$; whereas the second inequality requires
\begin{align*}
&0 < \mu - 6\delta - \frac{3}{4} \frac{3}{2-3\varepsilon} \sum_{k=1}^K |\bfsigma_k|^2 = \mu - 6\delta - \frac{2}{2-3\varepsilon} \Gamma_\bfsigma \\
\Leftrightarrow \qquad & \delta < \frac{1}{6}\left(\mu - \frac{2}{2-3\varepsilon} \Gamma_\bfsigma  \right) =  \frac{1}{6}\left(\mu - \Gamma_\bfsigma - \frac{3\varepsilon}{2-3\varepsilon} \Gamma_\bfsigma  \right).
\end{align*}
As long as the last term is positive, we can choose~$\delta> 0$ as a function of~$\varepsilon$. Thus, 
\begin{align*}
 0 < \frac{1}{6}\left(\mu - \Gamma_\bfsigma - \frac{3\varepsilon}{2-3\varepsilon} \Gamma_\bfsigma  \right) \qquad \Leftrightarrow \qquad \varepsilon < \frac{2}{3}\frac{\mu - \Gamma_\bfsigma}{\mu}. 
\end{align*}
We are ready to fix the parameters. We distinguish fixed parameters from free ones, by adding a star to them: firstly, we choose~$\varepsilon^* < \min\big\{\frac{2}{3}, \frac{2}{3}\frac{\mu - \Gamma_\bfsigma}{\mu}\big\}$ (and thus fix~$\beta^*$), where we use that the second term is positive thanks to the assumed Inequality~\eqref{ass:vis-dom-noise}; having~$\varepsilon^*$ fixed, we choose~$\delta^*< \min\big\{\frac{\varepsilon^*}{4},  \frac{1}{6}\big(\mu - \Gamma_\bfsigma - \frac{3\varepsilon^*}{2-3\varepsilon^*} \Gamma_\bfsigma  \big)\big\}$.

Going back to Inequality~\eqref{eq:pre-gronwall} and, additionally, neglecting the non-negative terms yield
\begin{align*}
&\frac{1}{2}\left( 1 - 6\delta^* \right)  \mathbb{E}\left[ \max_{1\leq n\leq M} \norm{\bfe_n}_{L^2_x}^2\right] \leq 3 (C + 2C(\delta^*)) (\Delta t)^{2\alpha}  + 3 C(\delta^*) \mathbb{E}\left[ \Delta t \sum_{m=1}^{M} \max_{1\leq n \leq m} \norm{\bfe_{n}}_{L^2_x}^2  \right].
\end{align*}
This means we can apply Gronwall's lemma, which then ensures the existence of a constant~$C$ such that 
\begin{align} \label{eq:post-Gronwall}
 \mathbb{E}\left[ \max_{1\leq n\leq M} \norm{\bfe_n}_{L^2_x}^2\right] \leq C (\Delta t)^{2\alpha}.
\end{align}
Using the above estimate in~\eqref{eq:pre-gronwall} establishes
\begin{align*}
&\frac{1}{2}\left( 1 - 6\delta^* \right)  \mathbb{E}\left[ \max_{1\leq n\leq M} \norm{\bfe_n}_{L^2_x}^2\right]+\frac{1}{2}\left( 1 - 6\delta^* - \frac{3}{2}\beta^* \right)  \mathbb{E}\left[ \sum_{m=0}^{M-1} \norm{\bfe_{m+1} - \bfe_{m}}_{L^2_x}^2 \right] \\
&\qquad \qquad +\left(\mu - 6\delta^* -  \frac{3}{4\beta^*}  \sum_{k=1}^K |\bfsigma_k|^2 \right) \mathbb{E}\left[ \Delta t \sum_{m=0}^{M-1} \norm{\nabla\bfe_{m+1}}_{L^2_x}^2 \right]  \leq C(\Delta t)^{2\alpha}.
\end{align*}
This concludes the proof of Theorem~\ref{thm:3.1}, provided the inequalities presented in~\eqref{eq:error-terms-all} are true. Now, we verify each inequality.

{\bf Ad I.} Term~$\mathrm{I}$ corresponds to the error introduced by freezing the analytic solution at a time instance instead of integrating over a time slice of size~$\Delta t$. Using Cauchy-Schwarz', Jensen's and weighted Young's inequalities, we infer that, for arbitrary~$\delta > 0$, 
\begin{align*}
\mathrm{I} &= \mathbb{E}\left[ \max_{1\leq n\leq M} \sum_{m=0}^{n-1} \int_{\mt} \mu \int_{t_{m}}^{t_{m+1}} \nabla \big(\bfu(t_{m+1})-\bfu(t)\big) \dt : \nabla \bfe_{m+1} \dx  \right] \\
&\leq  \mu \mathbb{E}\left[  \sum_{m=0}^{M-1}   \int_{t_{m}}^{t_{m+1}} \norm{\nabla \big(\bfu (t_{m+1})-\bfu(t)\big)}_{L^2_x} \dt  \, \norm{\nabla \bfe_{m+1} }_{L^2_x}   \right] \\
&\leq \delta \mathbb{E}\left[\Delta t \sum_{m=0}^{M-1} \norm{\nabla \bfe_{m+1} }_{L^2_x}^2 \right] + C(\delta) \mathbb{E}\left[\sum_{m=0}^{M-1}   \int_{t_{m}}^{t_{m+1}} \norm{\nabla \big(\bfu (t_{m+1})-\bfu(t)\big)}_{L^2_x}^2 \dt \right].
\end{align*}
Notice that, thanks to Corollary~\ref{cor:uholder} with $l = 2$, for arbitrary~$\alpha \in (0,1/2)$ there exists a constant $C$ such that
\begin{align*}
 \mathbb{E}\left[\sum_{m=0}^{M-1}   \int_{t_{m}}^{t_{m+1}} \norm{\nabla \big(\bfu (t_{m+1})-\bfu(t)\big)}_{L^2_x}^2 \dt \right] \leq C (\Delta t)^{2\alpha}. 
\end{align*}
Combining this with the previous estimate allows us to conclude, for arbitrary~$\delta > 0$, 
\begin{align*}
\mathrm{I} \leq \delta \mathbb{E}\left[\Delta t \sum_{m=0}^{M-1} \norm{\nabla \bfe_{m+1} }_{L^2_x}^2 \right]  + C(\delta) (\Delta t)^{2\alpha}.
\end{align*}

{\bf Ad II.} Term~$\mathrm{II}$ corresponds to the approximation of the non-linear convection. We denote the arithmetic mean of the analytic solution at integer times and errors by $\bfu^\star(t_{m+1/2}):=\frac{1}{2}(\bfu(t_{m+1})+\bfu(t_m))$ and $\bfe_{m+\frac{1}{2}}:= \tfrac{1}{2}(\bfe_{m+1}+ \bfe_{m})$, respectively. We start by artificially freezing the time integral, and adding and subtracting $\Delta t (\bfu^\star(t_{m+\frac{1}{2}}) \cdot \nabla) \bfu_{m+\frac{1}{2}} $:
\begin{align*}
I_2(m) &= \int_{t_{m}}^{t_{m+1}} \Big((\bfu^\star(t_{m+\frac{1}{2}})\cdot\nabla)\bfu^\star(t_{m+\frac{1}{2}})-(\bfu(t)\cdot\nabla)\bfu(t)\Big) \dt\\
&\quad-\Delta t \Big((\bfu^\star(t_{m+\frac{1}{2}})\cdot\nabla) \bfe_{m+\frac{1}{2}}+(\bfe_{m+\frac{1}{2}}\cdot\nabla)\bfu_{m+\frac{1}{2}}\Big).
\end{align*}
The first term solely depends on the analytic solution and measures the error for freezing it at a time instance; the second term is a time-discretized error of the analytic solution and its approximation. Consequently, 
\begin{align*}
\mathrm{II} &\leq \mathbb{E}\left[ \max_{1\leq n\leq M} \sum_{m=0}^{n-1} \int_{\mt} \int_{t_{m}}^{t_{m+1}} \Big((\bfu^\star(t_{m+\frac{1}{2}})\cdot\nabla)\bfu^\star(t_{m+\frac{1}{2}})-(\bfu(t)\cdot\nabla)\bfu(t)\Big) \dt \cdot \bfe_{m+1} \dx   \right] \\
&\quad + \mathbb{E}\left[ \max_{1\leq n\leq M} \sum_{m=0}^{n-1} \int_{\mt} -\Delta t \Big((\bfu^\star(t_{m+\frac{1}{2}})\cdot\nabla) \bfe_{m+\frac{1}{2}}+(\bfe_{m+\frac{1}{2}}\cdot\nabla)\bfu_{m+\frac{1}{2}}\Big) \cdot \bfe_{m+1} \dx   \right].
\end{align*} 
We will discuss both terms separately.

We start with the first term: Applying Cauchy-Schwarz', Jensen's and weighted Young's inequalities, we infer that, for arbitrary~$\delta > 0$, 
\begin{align*}
&\mathbb{E}\left[ \max_{1\leq n\leq M} \sum_{m=0}^{n-1} \int_{\mt} \int_{t_{m}}^{t_{m+1}} \Big((\bfu^\star(t_{m+\frac{1}{2}})\cdot\nabla)\bfu^\star(t_{m+\frac{1}{2}})-(\bfu(t)\cdot\nabla)\bfu(t)\Big) \dt \cdot \bfe_{m+1} \dx  \right] \\
&\qquad \leq \mathbb{E}\left[ \sum_{m=0}^{M-1} \int_{t_{m}}^{t_{m+1}} \norm{(\bfu^\star(t_{m+\frac{1}{2}})\cdot\nabla)\bfu^\star(t_{m+\frac{1}{2}})-(\bfu(t)\cdot\nabla)\bfu(t) }_{L^2_x} \dt \, \norm{\bfe_{m+1}}_{L^2_x} \right] \\
&\qquad \leq \delta \mathbb{E}\left[\max_{1\leq n\leq M} \norm{\bfe_{n} }_{L^2_x}^2   \right]  \\
&\qquad \qquad + C(\delta) \mathbb{E}\left[\left( \sum_{m=0}^{M-1} \int_{t_{m}}^{t_{m+1}} \norm{(\bfu^\star(t_{m+\frac{1}{2}})\cdot\nabla)\bfu^\star(t_{m+\frac{1}{2}})-(\bfu(t)\cdot\nabla)\bfu(t) }_{L^2_x} \dt \right)^2 \right].
\end{align*}
Notice that, using H\"older's inequality,
\begin{align*}
&\mathbb{E}\left[\left( \sum_{m=0}^{M-1} \int_{t_{m}}^{t_{m+1}} \norm{(\bfu^\star(t_{m+\frac{1}{2}})\cdot\nabla)\bfu^\star(t_{m+\frac{1}{2}})-(\bfu(t)\cdot\nabla)\bfu(t) }_{L^2_x} \dt \right)^2 \right] \\
&\leq 2 \mathbb{E}\left[\left( \sum_{m=0}^{M-1} \int_{t_{m}}^{t_{m+1}} \norm{([\bfu^\star(t_{m+\frac{1}{2}}) - \bfu(t) ]\cdot\nabla)\bfu^\star(t_{m+\frac{1}{2}})}_{L^2_x} \dt \right)^2 \right] \\
&\quad + 2\mathbb{E}\left[\left( \sum_{m=0}^{M-1} \int_{t_{m}}^{t_{m+1}} \norm{(\bfu(t)\cdot\nabla)[ \bfu^\star(t_{m+\frac{1}{2}}) - \bfu(t)] }_{L^2_x} \dt \right)^2 \right] \\
&\leq C (\Delta t)^{2\alpha} \left( \mathbb{E}\left[ \norm{\bfu}_{C^\alpha([0,T]; L^2_x)}^4 \right] \mathbb{E}\left[ \norm{\nabla \bfu}_{C([0,T];L^\infty_x)}^4 \right] \right)^{1/2} \\
&\quad+ C (\Delta t)^{2\alpha} \left( \mathbb{E}\left[ \norm{\nabla \bfu}_{C^\alpha([0,T]; L^2_x)}^4 \right] \mathbb{E}\left[ \norm{ \bfu}_{C([0,T];L^\infty_x)}^4 \right] \right)^{1/2}.
\end{align*}
It remains to argue, why the analytic solution is sufficiently regular: Firstly, Corollary~\ref{cor:uholder} applied with $l=2$ guarantees finiteness of the H\"older norms; and secondly, the pathwise estimate~\eqref{eq:enl} invoked with $l = 3$, together with the embedding $W^{3,2}(\mt)\hookrightarrow W^{1,\infty}(\mt)$, ensure that $\mathbb{E}\big[ \norm{\nabla \bfu}_{C([0,T];L^\infty_x)}^4 \big] \leq C$.

Next, we analyse the second term: Applying H\"older's inequality, we derive that
\begin{align*}
&\mathbb{E}\left[ \max_{1\leq n\leq M} \sum_{m=0}^{n-1} \int_{\mt} -\Delta t \Big((\bfu^\star(t_{m+\frac{1}{2}})\cdot\nabla) \bfe_{m+\frac{1}{2}}+(\bfe_{m+\frac{1}{2}}\cdot\nabla)\bfu_{m+\frac{1}{2}}\Big) \cdot \bfe_{m+1} \dx   \right] \\
&\qquad \leq \mathbb{E}\left[ \sum_{m=0}^{M-1} \Delta t \norm{\bfu^\star(t_{m+\frac{1}{2}})}_{L^\infty_x} \norm{ \nabla \bfe_{m+\frac{1}{2}} }_{L^2_x} \norm{\bfe_{m+1}}_{L^2_x}  \right] \\
&\qquad \qquad  + \mathbb{E}\left[  \sum_{m=0}^{M-1} \Delta t  \norm{\bfe_{m+\frac{1}{2}} }_{L^2_x} \norm{ \nabla \bfu_{m+\frac{1}{2}}}_{L^\infty_x} \norm{\bfe_{m+1}}_{L^2_x}  \right].
\end{align*}
Using weighted Young's inequality and~$\norm{\nabla \bfe_{m+1/2}}_{L^2_x}\leq \tfrac{1}{2}(\norm{\nabla \bfe_{m+1}}_{L^2_x} +\norm{\nabla \bfe_{m}}_{L^2_x} )$, we find that, for arbitrary~$\delta > 0$,
\begin{align*}
&\mathbb{E}\left[ \sum_{m=0}^{M-1} \Delta t \norm{\bfu^\star(t_{m+\frac{1}{2}})}_{L^\infty_x} \norm{ \nabla \bfe_{m+\frac{1}{2}} }_{L^2_x} \norm{\bfe_{m+1}}_{L^2_x}  \right] \\
&\qquad \leq \delta  \mathbb{E}\left[ \sum_{m=0}^{M-1} \Delta t  \norm{ \nabla \bfe_{m+1} }_{L^2_x}^2  \right] +  C(\delta) \mathbb{E}\left[ \sum_{m=0}^{M-1} \Delta t \norm{\bfu^\star(t_{m+\frac{1}{2}})}_{L^\infty_x}^2 \norm{\bfe_{m+1}}_{L^2_x}^2  \right],
\end{align*}
Recall that the analytic solution is bounded $\mathbb{P}$-a.s., thanks to Inequality~\eqref{eq:enl} applied with~$l = 2$ and the embedding $W^{2,2}(\mt)\hookrightarrow L^{\infty}(\mt)$. Thus, additionally shifting the summation index and estimating the local time error by the maximal error of the past, results in
\begin{align*}
&\mathbb{E}\left[ \sum_{m=0}^{M-1} \Delta t \norm{\bfu^\star(t_{m+\frac{1}{2}})}_{L^\infty_x}^2 \norm{\bfe_{m+1}}_{L^2_x}^2  \right]\\
&\qquad \leq C \mathbb{E}\left[ \sum_{m=1}^{M} \Delta t \norm{\bfe_{m}}_{L^2_x}^2  \right] \leq C \mathbb{E}\left[ \sum_{m=1}^{M} \Delta t \max_{1\leq n \leq m} \norm{\bfe_{n}}_{L^2_x}^2  \right].
\end{align*}
Similarly, one checks that 
\begin{align*}
\mathbb{E}\left[  \sum_{m=0}^{M-1} \Delta t  \norm{\bfe_{m+\frac{1}{2}} }_{L^2_x} \norm{ \nabla \bfu_{m+\frac{1}{2}}}_{L^\infty_x} \norm{\bfe_{m+1}}_{L^2_x}  \right] \leq  C \mathbb{E}\left[ \sum_{m=1}^{M} \Delta t \max_{1\leq n \leq m} \norm{\bfe_{n}}_{L^2_x}^2  \right],
\end{align*}
where one uses the pathwise stability of the algorithm for arbitrary high spatial norms, i.e., Inequality~\eqref{lem:3.1b} with~$l=3$ and the embedding $W^{3,2}(\mt)\hookrightarrow W^{1,\infty}(\mt)$.

In total, we have shown that, for arbitrary~$\delta > 0$, 
\begin{align*}
\mathrm{II} \leq C (\Delta t)^{2\alpha} +  \delta  \mathbb{E}\left[ \sum_{m=0}^{M-1} \Delta t  \norm{ \nabla \bfe_{m+1} }_{L^2_x}^2  \right] +  C(\delta) \mathbb{E}\left[ \sum_{m=1}^{M} \Delta t \max_{1\leq n \leq m} \norm{\bfe_{n}}_{L^2_x}^2  \right].
\end{align*}

{\bf Ad III.} Term~$\mathrm{III}$ corresponds to the approximation of the Stratonovich integral. We start with the  decomposition $I_3(m) = I_3^a(m) + I_3^b(m)$, where 
\begin{align*}
I_3^a(m) &:= \sum_{k=1}^K\left(\int_{t_{m}}^{t_{m+1}}(\bfsigma_k\cdot \nabla )\bfu \circ\,\dd W_k-(\bfsigma_k\cdot\nabla )\bfu^\star(t_{m+\frac{1}{2}}) \Delta_m W_k\right), \\
I_3^b(m) &:=\sum_{k=1}^K (\bfsigma_k\cdot\nabla )\bfe_{m+\frac{1}{2}}\Delta_m W_k.
\end{align*}
Consequently,
\begin{align*}
\mathrm{III} &\leq 
\mathbb{E}\left[ \max_{1\leq n\leq M} \sum_{m=0}^{n-1} \int_{\mt} I_3^a(m) \cdot \bfe_{m+1} \dx   \right] + \mathbb{E}\left[ \max_{1\leq n\leq M} \sum_{m=0}^{n-1} \int_{\mt} I_3^b(m) \cdot \bfe_{m+1} \dx   \right] \\
&=: \mathrm{III}_a + \mathrm{III}_b.
\end{align*}

{\bf Ad III$_\mathbf{a}$.} We decompose~$I_3^a(m)$. Using the definition of the Stratonovich integral in terms of the It\^o integral (see Equation~\eqref{eq:Strato-and-Ito}), which holds as an identity on $L^2(\mathbb{T}^2)$ since $\bfu$ is sufficiently regular; together with adding and subtracting~$(\bfsigma_k\cdot\nabla )\bfu(t_{m}) \Delta_m W_k$, we obtain 
\begin{align*}
I_3^a(m) &= \underbrace{\sum_{k=1}^K \int_{t_{m}}^{t_{m+1}}(\bfsigma_k\cdot \nabla )\big(\bfu - \bfu(t_m) \big) \, \dd W_k}_{=:R_1(m)}\\
&\quad + \underbrace{\sum_{k=1}^K\frac{1}{2}  \bigg(  \int_{t_{m}}^{t_{m+1}}(\bfsigma_k\cdot \nabla )[(\bfsigma_k\cdot \nabla )\bfu] \, \dd t-(\bfsigma_k\cdot\nabla) (\bfu(t_{m+1}) - \bfu(t_m))\Delta_m W_k\bigg)}_{=:\mathfrak{A}(m)}.
\end{align*}
We rewrite the Term~$\mathfrak{A}(m)$. For fixed~$\bfvarphi \in C^\infty_{\Div}(\mathbb{T}^2)$, after invoking integration by parts and interchanging spatial and temporal integrals, we find that
\begin{align} \nonumber
\int_{\mt} \mathfrak{A}(m) \cdot \bfphi \, \dd x &= \sum_{k=1}^K\frac{1}{2}   \int_{\mt} (\bfu(t_{m+1}) - \bfu(t_m))  \cdot (\bfsigma_k\cdot \nabla ) \bfvarphi \dx \, \Delta_m W_k \\ \label{eq:Am-second}
&\qquad +\int_{\mt} \sum_{k=1}^K\frac{1}{2}  \int_{t_{m}}^{t_{m+1}}   (\bfsigma_k\cdot \nabla )[(\bfsigma_k\cdot \nabla ) \bfu] \, \dd t \cdot \bfvarphi \dx  .
\end{align}
Notice that~$(\bfsigma_k\cdot \nabla)\bfvarphi \in  C^\infty_{\Div}(\mathbb{T}^2)$, therefore utilising that $\bfu$ is a solution to Equation~\eqref{eq:mmnt}, the first term can be rewritten as 
\begin{align*}
&\sum_{k=1}^K\frac{1}{2} \int_{\mt}  \left( \bfu(t_{m+1}) - \bfu(t_m) \right)  \cdot (\bfsigma_k\cdot \nabla ) \bfvarphi \dx \, \Delta_m W_k\\
&= \sum_{k=1}^K\frac{1}{2} \left( \int_{t_m}^{t_{m+1}} \int_{\mt}\bfu\otimes\bfu:\nabla[(\bfsigma_k\cdot \nabla ) \bfvarphi] \dx\,\dif t \right) \, \Delta_m W_k\\
&\quad + \sum_{k=1}^K\frac{1}{2} \left( -\mu \int_{t_m}^{t_{m+1}} \int_{\mt}\nabla\bfu:\nabla [(\bfsigma_k\cdot \nabla ) \bfvarphi]\dx\,\dif t \right) \, \Delta_m W_k \\
&\quad + \sum_{k=1}^K\frac{1}{2}   \left( -\sum_{\ell=1}^K \int_{t_m}^{t_{m+1}} \int_{\mt}\bfu\otimes \bfsigma_{\ell} :\nabla[(\bfsigma_k\cdot \nabla )] \bfvarphi \,\dx\,\D W_{\ell} \right)  \, \Delta_m W_k \\
&\quad + \sum_{k=1}^K\frac{1}{2} \left( -\frac{1}{2}\sum_{\ell=1}^K \int_{t_m}^{t_{m+1}} \int_{\mt}\bfsigma_{\ell}\otimes\bfsigma_{\ell} \nabla\bfu:\nabla[(\bfsigma_k\cdot \nabla ) \bfvarphi] \dx\,\dt  \right) \, \Delta_m W_k.
\end{align*}
Integrating by parts together with interchanging spatial and temporal integrals result in
\begin{align*}
&\sum_{k=1}^K\frac{1}{2} \int_{\mt}  \left( \bfu(t_{m+1}) - \bfu(t_m) \right)  \cdot (\bfsigma_k\cdot \nabla ) \bfvarphi \dx \, \Delta_m W_k\\
&= \int_{\mt} \underbrace{\sum_{k=1}^K\frac{1}{2} \left(  \int_{t_m}^{t_{m+1}} (\bfsigma_k\cdot \nabla ) [(\bfu\cdot \nabla )\bfu] \,\dif t \right) \, \Delta_m W_k}_{=:R_2(m)} \cdot\, \bfvarphi \dx\\
&\quad + \int_{\mt} \underbrace{\sum_{k=1}^K\frac{1}{2} \left( -\mu \int_{t_m}^{t_{m+1}} (\bfsigma_k\cdot \nabla ) \Delta \bfu \,\dif t \right) \, \Delta_m W_k }_{=: R_3(m) } \cdot \, \bfvarphi \dx \\
&\quad + \int_{\mt} \underbrace{\sum_{k=1}^K\frac{1}{2} \left( -\frac{1}{2}\sum_{\ell=1}^K \int_{t_m}^{t_{m+1}} (\bfsigma_k\cdot \nabla )  \big[(\bfsigma_{\ell} \cdot \nabla)[ (\bfsigma_{\ell} \cdot \nabla)\bfu ] \big] \,\dt  \right) \, \Delta_m W_k }_{=: R_4(m)} \cdot \, \bfvarphi \dx \\
&\quad + \int_{\mt} \sum_{k=1}^K\frac{1}{2}   \left( -\sum_{\ell=1}^K \int_{t_m}^{t_{m+1}}  (\bfsigma_k\cdot \nabla )[(\bfsigma_{\ell} \cdot \nabla )\bfu]  \,\D W_{\ell} \right)  \, \Delta_m W_k \cdot \, \bfvarphi \,\dx.
\end{align*}
Due to density of $C^\infty_{\Div}(\mathbb{T}^2) \subset L^{2}_{\Div}(\mathbb{T}^2)$, above equality extends to all $\bfvarphi \in  L^{2}_{\Div}(\mathbb{T}^2)$.

All but the last term on the right-hand-side enjoy the desired time decay with respect to the time resolution by standard arguments. We make this precise later. But first, we further decompose the last term: using the second term on the right-hand-side of Equation~\eqref{eq:Am-second} to compensate the last one, we decompose as follows:
\begin{align*} 
&\int_{\mt} \sum_{k=1}^K\frac{1}{2}   \left( -\sum_{\ell=1}^K \int_{t_m}^{t_{m+1}}  (\bfsigma_k\cdot \nabla )[(\bfsigma_{\ell} \cdot \nabla )\bfu]  \,\D W_{\ell} \right)  \, \Delta_m W_k \cdot  \bfvarphi \,\dx \\
&\qquad \qquad +\int_{\mt} \sum_{k=1}^K\frac{1}{2}  \int_{t_{m}}^{t_{m+1}}   (\bfsigma_k\cdot \nabla )[(\bfsigma_k\cdot \nabla ) \bfu] \, \dd t \cdot \bfvarphi \dx\\
&\qquad= \int_{\mt} R_5(m) \cdot \bfvarphi \,\dx  + \int_{\mt}  R_6(m) \cdot \bfvarphi \,\dx + \int_{\mt}  R_7(m)\cdot \bfvarphi\,\dx,
 \end{align*}
where
\begin{align*}
   R_5(m)&:=-\frac{1}{2}\sum_{k=1}^K\sum_{\ell=1}^K \left( \int_{t_m}^{t_{m+1}}(\bfsigma_k\cdot\nabla)[(\bfsigma_\ell \cdot \nabla )(\bfu -\bfu(t_m))]\dd W_\ell\right)\Delta_m W_k,\\
    R_6(m)&:=-\frac{1}{2}\sum_{k=1}^K\sum_{\ell=1}^K(\bfsigma_k\cdot\nabla)[(\bfsigma_\ell \cdot \nabla )\bfu (t_m)]\Delta_m W_k \Delta_m W_\ell+ \frac{1}{2}\Delta t\sum_{k=1}^K(\bfsigma_k\cdot\nabla)[(\bfsigma_k\cdot \nabla )\bfu (t_m)],\\
R_7(m)&:=\frac{1}{2}\sum_{k=1}^K\int_{t_{m}}^{t_{m+1}}(\bfsigma_k\cdot\nabla)[(\bfsigma_k\cdot \nabla )(\bfu-\bfu(t_m))]\,\dd t.
\end{align*}
The Terms~$R_5(m)$ and~$ R_7(m)$ enjoy accelerated time convergence since they depend on time increments of the analytic solution. The Term~$R_6(m)$ cannot utilise solution increments. Instead we notice that $R_6(m)$ is a centred $L^2_x$-valued random variable. We will use martingale inequalities to estimate this term eventually. 

Before we start analysing each term, we summarize the decomposition of~$I_3^a(m)$ we have obtained so far:
\begin{align*}
\mathbb{E}\left[ \max_{1\leq n\leq M} \sum_{m=0}^{n-1} \int_{\mt} I_3^a(m) \cdot \bfe_{m+1} \dx   \right] \leq \sum_{\eta = 1}^7 \mathbb{E}\left[ \max_{1\leq n\leq M} \sum_{m=0}^{n-1} \int_{\mt} R_\eta(m) \cdot \bfe_{m+1} \dx   \right].
\end{align*} 
Next, we derive estimates for each term. 

Let $\eta \in \{1,\ldots,7\}\backslash \{1,6\}$ (the Cases~$\eta = 1$ and~$\eta = 6$ are special and will be considered afterwards). H\"older's and weighted Young's inequalities show, for arbitrary~$\delta > 0$,  
\begin{align*}
&\mathbb{E}\left[ \max_{1\leq n\leq M} \sum_{m=0}^{n-1} \int_{\mt} R_\eta(m) \cdot \bfe_{m+1} \dx   \right] \leq \mathbb{E}\left[ \sum_{m=0}^{M-1} \norm{R_\eta(m)}_{L^2_x} \norm{\bfe_{m+1}}_{L^2_x}   \right] \\
&\qquad \leq \delta \mathbb{E}\left[ \max_{1\leq n \leq M} \norm{\bfe_{n}}_{L^2_x}^2 \right] + C(\delta) \mathbb{E}\left[ \frac{1}{\Delta t} \sum_{m=0}^{M-1} \norm{R_\eta(m)}_{L^2_x}^2  \right].
\end{align*}
H\"older's inequality guarantees
\begin{align*}
&\mathbb{E}\left[ \norm{R_2(m)}_{L^2_x}^2 \right] = \mathbb{E}\left[ \bigg\|\sum_{k=1}^K\frac{1}{2} \left(  \int_{t_m}^{t_{m+1}} (\bfsigma_k\cdot \nabla ) [(\bfu\cdot \nabla )\bfu] \,\dif t \right) \, \Delta_m W_k \bigg\|_{L^2_x}^2 \right]  \\
&\quad \leq \frac{K}{4}\sum_{k=1}^K \mathbb{E}\left[ \bigg\| \int_{t_m}^{t_{m+1}} (\bfsigma_k\cdot \nabla ) [(\bfu\cdot \nabla )\bfu] \,\dif t \bigg\|_{L^2_x}^2 \, \abs{\Delta_m W_k}^2  \right]  \\
&\quad \leq  \Delta t\,\frac{K}{4}\sum_{k=1}^K \mathbb{E}\left[ \int_{t_m}^{t_{m+1}} \norm{(\bfsigma_k\cdot \nabla ) [(\bfu\cdot \nabla )\bfu] }_{L^2_x}^2 \,\dif t \, \abs{\Delta_m W_k}^2  \right] \\
&\quad \leq  \Delta t\,\frac{K}{4}\sum_{k=1}^K \left( \mathbb{E}\left[ \left( \int_{t_m}^{t_{m+1}} \norm{(\bfsigma_k\cdot \nabla ) [(\bfu\cdot \nabla )\bfu] }_{L^2_x}^2 \,\dif t\right)^2  \right] \right)^{1/2} \left( \mathbb{E}\left[ \abs{\Delta_m W_k}^4  \right] \right)^{1/2}.
\end{align*}
 Since $(\Delta_m W_k)_{m,k}$, $m \in \{0,\ldots,M-1\}$ and $k \in \{1, \ldots, K\}$, is a family of identically distributed random variables (centred normal distribution with variance~$\Delta t$), we conclude that $\left(\E \left[ |\Delta_m W_k|^4\right]\right)^{1/2} \leq C \Delta t$. Additionally, by computing derivatives explicitly, together with using H\"older's inequality and Young's inequality, we find that 
\begin{align*}
&\mathbb{E}\left[ \left( \int_{t_m}^{t_{m+1}} \norm{(\bfsigma_k\cdot \nabla ) [(\bfu\cdot \nabla )\bfu] }_{L^2_x}^2 \,\dif t\right)^2  \right] \\
&\qquad \leq C \mathbb{E}\left[ \left( \int_{t_m}^{t_{m+1}} \norm{\nabla \bfu}_{L^4_x}^4 + \norm{\bfu}_{L^\infty_x}^2 \norm{\nabla^2 \bfu}_{L^2_x}^2 \,\dif t\right)^2  \right] \\
&\qquad \leq C (\Delta t)^2 \left( \mathbb{E}\left[ \norm{\nabla \bfu}_{C([0,T];L^4_x)}^8 \right] + \mathbb{E}\left[ \norm{\bfu}_{C([0,T];L^\infty_x)}^8 \right] + \mathbb{E}\left[ \norm{\nabla^2 \bfu}_{C([0,T];L^2_x)}^8 \right] \right).
\end{align*}
Due to Inequality~\eqref{eq:enl} (applied with~$l = 2$) and the embedding~$W^{2,2}(\mathbb{T}^2) \hookrightarrow L^\infty(\mathbb{T}^2)$, the analytic solution possesses enough regularity. This and the previous estimates imply
\begin{align*}
\mathbb{E}\left[ \norm{R_2(m)}_{L^2_x}^2 \right] \leq C (\Delta t)^3.
\end{align*}
Analogously, one obtains similar estimates for the other cases, so that in total:
\begin{align*}
 \sum_{\eta=1,\, \eta \neq 1,6}^7\mathbb{E}\left[ \frac{1}{\Delta t} \sum_{m=0}^{M-1} \norm{R_\eta(m)}_{L^2_x}^2  \right] &\leq C \Delta t.
\end{align*}

It remains to check the Cases~$\eta = 1$ and~$\eta = 6$. Both rely on the same argument: a martingale inequality. Before we can use a martingale estimate, we need to delay the evaluation time of the error from~$t_{m+1}$ to~$t_m$. For $\eta \in\{1,6\}$, we estimate
\begin{align*}
&\mathbb{E}\left[ \max_{1\leq n\leq M} \sum_{m=0}^{n-1} \int_{\mt} R_\eta(m) \cdot \bfe_{m+1} \dx \right] \\
&\leq \mathbb{E}\left[ \max_{1\leq n\leq M} \sum_{m=0}^{n-1} \int_{\mt} R_\eta(m) \cdot \bfe_{m} \dx   \right]  + \mathbb{E}\left[ \max_{1\leq n\leq M} \sum_{m=0}^{n-1} \int_{\mt} R_\eta(m) \cdot (\bfe_{m+1} - \bfe_{m} ) \dx   \right].
\end{align*}
Now, the first term is well-prepared for the application of the Burkholder-Davis-Gundy's inequality, while the second term can be handled with standard arguments. 

{\underline{$\eta = 1$}.}
We start with the martingale part: let us verify that
\begin{align} \label{eq:martingale-R1}
\{1,\ldots,M\} \ni n \mapsto  \sum_{m=0}^{n-1}  \int_{\mt} R_1(m) \cdot \bfe_{m} \dx 
\end{align}
is indeed a real-valued martingale with respect to the filtration generated by the Wiener increments up to time~$t_m$, i.e.,~$\mathfrak{F}_{t_m}= \sigma\left( W_k(t) | t \leq t_m, k \in \{1,\ldots, K\}  \right)$. To shorten the notation, we write~$\mathfrak{F}_{m}$ instead of $\mathfrak{F}_{t_m}$. 

Let $s \in \{1,\ldots, M\}$ such that $s > n$. Then, using the linearity of conditional expectations, $\mathfrak{F}_{m+1}$-measurability of~$R_1(m)$, $\mathfrak{F}_{m}$-measurability of~$\bfe_m$ and the tower property of conditional expectations
\begin{align*}
&\mathbb{E}\left[\sum_{m=0}^{n-1}  \int_{\mt} R_1(m) \cdot \bfe_{m} \dx \bigg| \mathfrak{F}_s \right] \\
&\qquad = \sum_{m=0}^{s-1} \mathbb{E}\left[  \int_{\mt} R_1(m) \cdot \bfe_{m} \dx \bigg| \mathfrak{F}_s \right] +  \sum_{m=s}^{n-1} \mathbb{E}\left[  \int_{\mt} R_1(m) \cdot \bfe_{m} \dx \bigg| \mathfrak{F}_s \right] \\
&\qquad= \sum_{m=0}^{s-1}  \int_{\mt} R_1(m) \cdot \bfe_{m} \dx  +  \sum_{m=s}^{n-1} \mathbb{E}\left[  \int_{\mt} \mathbb{E}\left[  R_1(m) \big| \mathfrak{F}_{m} \right] \cdot \bfe_{m} \dx  \bigg| \mathfrak{F}_s \right].
\end{align*}
Lastly, notice that $ \mathbb{E}\left[  R_1(m) \big| \mathfrak{F}_{m} \right] = 0$, since $R_1$ is defined in terms of an It\^o-integral. This implies that~\eqref{eq:martingale-R1} defines a $(\mathfrak{F}_m)_{m=1}^M$-martingale.

Therefore, applying Burkholder-Davis-Gundy's inequality yields
\begin{align*}
\mathbb{E}\left[ \max_{1\leq n\leq M} \sum_{m=0}^{n-1} \int_{\mt} R_1(m) \cdot \bfe_{m} \dx   \right] \leq C \mathbb{E}\left[ \left( \sum_{m=0}^{M-1} \left( \int_{\mt} R_1(m) \cdot \bfe_{m} \dx \right)^2    \right)^{1/2} \right].
\end{align*}
We continue with an application of H\"older's and weighted Young's inequalities, which show, for arbitrary~$\delta > 0$,
\begin{align*}
&\mathbb{E}\left[ \left( \sum_{m=0}^{M-1} \left( \int_{\mt} R_1(m) \cdot \bfe_{m} \dx \right)^2    \right)^{1/2} \right]  \leq \delta \mathbb{E}\left[ \max_{1\leq n \leq M} \norm{\bfe_{n}}_{L^2_x}^2 \right] + C(\delta) \mathbb{E}\left[ \sum_{m=0}^{M-1} \norm{R_1(m)}_{L^2_x}^2 \right].
\end{align*}

Next, we investigate the remaining term: H\"older's and weighted Young's inequalities imply, for arbitrary~$\delta > 0$,
\begin{align*}
&\mathbb{E}\left[ \max_{1\leq n\leq M} \sum_{m=0}^{n-1} \int_{\mt} R_1(m) \cdot (\bfe_{m+1} - \bfe_{m} ) \dx   \right] \\
&\qquad \leq \delta \mathbb{E}\left[  \sum_{m=0}^{M-1} \norm{\bfe_{m+1} - \bfe_{m}}_{L^2_x}^2  \right] + C(\delta) \mathbb{E}\left[  \sum_{m=0}^{M-1} \norm{R_1(m)}_{L^2_x}^2  \right].
\end{align*}
Applying It\^o's isometry, we find that
\begin{align*}
 \mathbb{E}\left[  \sum_{m=0}^{M-1} \norm{R_1(m)}_{L^2_x}^2  \right] &=  \mathbb{E}\left[  \sum_{m=0}^{M-1} \sum_{k=1}^K \int_{t_{m}}^{t_{m+1}} \norm{ (\bfsigma_k\cdot \nabla )\big(\bfu - \bfu(t_m) \big)}_{L^2_x}^2 \dd t  \right].
\end{align*}
Therefore, using Corollary~\ref{cor:uholder} applied with~$l=2$, we infer that
\begin{align*}
 \mathbb{E}\left[  \sum_{m=0}^{M-1} \norm{R_1(m)}_{L^2_x}^2  \right] \leq C (\Delta t)^{2\alpha} \mathbb{E}\left[ \norm{\nabla \bfu}_{C^\alpha([0,T];L^2_x)}^2 \right] \leq C (\Delta t)^{2\alpha} .
\end{align*}
In total, we have shown that, for arbitrary~$\delta >0$, 
\begin{align*}
&\mathbb{E}\left[ \max_{1\leq n\leq M} \sum_{m=0}^{n-1} \int_{\mt} R_1(m) \cdot \bfe_{m+1} \dx \right] \\
&\qquad \leq \delta \mathbb{E}\left[  \sum_{m=0}^{M-1} \norm{\bfe_{m+1} - \bfe_{m}}_{L^2_x}^2  \right] + \delta \mathbb{E}\left[ \max_{1\leq n \leq M} \norm{\bfe_{n}}_{L^2_x}^2 \right]  + C(\delta) (\Delta t)^{2\alpha}.
\end{align*}

{\underline{$\eta = 6$}.} Similar to previous case, one checks that 
\begin{align*}
\{1,\ldots,M\} \ni n \mapsto  \sum_{m=0}^{n-1}  \int_{\mt} R_6(m) \cdot \bfe_{m} \dx 
\end{align*}
is a $(\mathfrak{F}_m)_{m=1}^M$-martingale. This can be done by utilising the identity: $\E[\Delta W_k\Delta_m W_l]=\Delta t \, \delta_{kl}$, where $\delta_{kl}$ denotes the Kronecker delta function, i.e., $\delta_{kl} = 1$ if $k=l$ and $\delta_{kl}=0$ if $k\neq l$; the tower property of conditional expectations; and the independence of Wiener increments. Therefore, we can follow the same argumentation as in the Case~$\eta=1$ to deduce
\begin{align*}
&\mathbb{E}\left[ \max_{1\leq n\leq M} \sum_{m=0}^{n-1} \int_{\mt} R_6(m) \cdot \bfe_{m+1} \dx \right] \\
&\qquad \leq \delta \mathbb{E}\left[  \sum_{m=0}^{M-1} \norm{\bfe_{m+1} - \bfe_{m}}_{L^2_x}^2  \right] + \delta \mathbb{E}\left[ \max_{1\leq n \leq M} \norm{\bfe_{n}}_{L^2_x}^2 \right]  + C(\delta) \mathbb{E}\left[  \sum_{m=0}^{M-1} \norm{R_6(m)}_{L^2_x}^2  \right].
\end{align*}
Lastly, using H\"older's inquality and moment estimates for Wiener increments, it holds
\begin{align*}
\mathbb{E}\left[  \sum_{m=0}^{M-1} \norm{R_6(m)}_{L^2_x}^2  \right] &\leq \frac{K^2}{2} \sum_{k=1}^K\sum_{\ell=1}^K \mathbb{E}\left[  \sum_{m=0}^{M-1} \norm{(\bfsigma_k\cdot\nabla)[(\bfsigma_\ell \cdot \nabla )\bfu (t_m)]}^2 \abs{\Delta_m W_k \Delta_m W_\ell}^2  \right] \\
&\quad +\frac{K}{2}  \sum_{k=1}^K \mathbb{E}\left[  \sum_{m=0}^{M-1} \norm{(\bfsigma_k\cdot\nabla)[(\bfsigma_k\cdot \nabla )\bfu (t_m)]}_{L^2_x}^2 (\Delta t)^2 \right]\\
&\leq C (\Delta t) \left( \mathbb{E}\left[ \norm{\nabla^2 \bfu}_{C([0,T];L^2_x)}^4 \right]\right)^{1/2} \leq C (\Delta t),
\end{align*}
where the last estimate follows from Inequality~\eqref{eq:enl} with~$l=2$.

At this point, we have finished the investigation of the first error contribution of~$\mathrm{III}$, which in summary reads: for arbitrary~$\delta >0$,
\begin{align}
\mathrm{III}_a \leq \delta \mathbb{E}\left[  \sum_{m=0}^{M-1} \norm{\bfe_{m+1} - \bfe_{m}}_{L^2_x}^2  \right] + \delta \mathbb{E}\left[ \max_{1\leq n \leq M} \norm{\bfe_{n}}_{L^2_x}^2 \right]  + C(\delta) (\Delta t)^{2\alpha}.
\end{align} 

{\bf Ad III$_\mathbf{b}$.} We shortly recall the definition of~$\mathrm{III}_b$:
\begin{align*}
\mathrm{III}_b &= \mathbb{E}\left[ \max_{1\leq n\leq M} \sum_{m=0}^{n-1} \int_{\mt} \sum_{k=1}^K (\bfsigma_k\cdot\nabla )\bfe_{m+\frac{1}{2}}\Delta_m W_k\cdot \bfe_{m+1} \dx   \right].
\end{align*}
Utilising $\int_{\mt} (\bfsigma_k \cdot \nabla)\bfa \cdot \bfa \,\dd x = 0$ for $\bfa \in \{ \bfe_{m},\bfe_{m+1}\}$, we find that
\begin{align*}
\mathrm{III}_b &= \frac{1}{2} \mathbb{E}\left[ \max_{1\leq n\leq M} \sum_{m=0}^{n-1} \int_{\mt} \sum_{k=1}^K (\bfsigma_k\cdot\nabla )\bfe_{m}\Delta_m W_k\cdot (\bfe_{m+1} -\bfe_{m})  \dx   \right]. 
\end{align*}
An application of H\"older's and weighted Young's inequalities, together with the tower property of conditional expectations and the independence of Wiener increments, imply (for arbitrary~$\beta >0$)
\begin{align*}
\mathrm{III}_b &\leq \frac{1}{2} \mathbb{E}\left[  \sum_{m=0}^{M-1} \bigg\| \sum_{k=1}^K (\bfsigma_k\cdot\nabla )\bfe_{m}\Delta_m W_k \bigg\|_{L^2_x} \norm{\bfe_{m+1} - \bfe_{m}}_{L^2_x}  \right] \\
&\leq \frac{\beta}{4}  \mathbb{E}\left[  \sum_{m=0}^{M-1} \norm{\bfe_{m+1} - \bfe_{m}}_{L^2_x}^2  \right] +  \frac{1}{4 \beta}  \mathbb{E}\left[  \sum_{m=0}^{M-1} \bigg\| \sum_{k=1}^K (\bfsigma_k\cdot\nabla )\bfe_{m}\Delta_m W_k \bigg\|_{L^2_x}^2  \right] \\
&=  \frac{\beta}{4}  \mathbb{E}\left[  \sum_{m=0}^{M-1} \norm{\bfe_{m+1} - \bfe_{m}}_{L^2_x}^2  \right] +  \frac{1}{4 \beta}  \mathbb{E}\left[  \Delta t  \sum_{m=0}^{M-1} \sum_{k=1}^K \norm{(\bfsigma_k\cdot\nabla )\bfe_{m}}_{L^2_x}^2  \right].
\end{align*}
Applying H\"older's inequality once more to the second term, we arrive at the final estimate:
\begin{align*}
\mathrm{III}_b \leq \frac{\beta}{4}  \mathbb{E}\left[  \sum_{m=0}^{M-1} \norm{\bfe_{m+1} - \bfe_{m}}_{L^2_x}^2  \right] +  \frac{1}{4 \beta}  \sum_{k=1}^K |\bfsigma_k|^2  \mathbb{E}\left[   \Delta t  \sum_{m=0}^{M-1}  \norm{\nabla \bfe_{m}}_{L^2_x}^2\right].
\end{align*}
As indicated in the arguments below \eqref{eq:pre-gronwall},
Assumption~\eqref{ass:vis-dom-noise} allows us to absorb both terms
to conclude the proof.

\end{proof}

\section{Numerical simulations} \label{sec:numerical-experiments}
In this section, we conduct two sets of numerical experiments: the first set of experiments~(SOE--1) is motivated by the use of different noises for the modelling of stochastic fluids. While the convergence analysis of our algorithm for additive and multiplicative noises is left open, we numerically explore the effect of these noise types on the kinetic energy; the second set of experiments~(SOE--2) generates numerical evidence on possible extensions of Theorem~\ref{thm:3.1} to fully discrete algorithms as well as relaxed data assumptions. The main research questions are:
\begin{enumerate}
\item Can the torus~$\mathbb{T}^2$ be replaced by a Lipschitz domain~$\mathcal{D}\subset \mathbb{R}^2$? 
\item Can periodic boundary conditions be replaced by \textit{no-slip} boundary conditions? 
\item Can constant transport fields $\bfsigma_k$, $k=1, \ldots,K$, be replaced by non-constant ones?
\end{enumerate}
We implement a fully discrete algorithm based on the time-discretisation~\eqref{tdiscr}, and combined with the mixed Finite Element Method with anti-symmetrised (non-)linear convection. We test the algorithm on various parameter configurations, alternating initial conditions and transport fields. We monitor the sensitivity of the algorithm with respect to the temporal discretisation, while keeping the spatial discretisation fixed. This allows us to determine the algorithm's temporal rate of convergence experimentally.

The implementation of the algorithm as well as the code used for conducting the experiments is available at~\href{https://github.com/joernwichmann/NSE-Transport}{https://github.com/joernwichmann/NSE-Transport}; it uses the open-source finite element package \textit{Firedrake}~\cite{FiredrakeUserManual}, which itself heavily relies on \textit{PETSc}~\cite{petsc-web-page}.

The section is structured as follows: we first introduce the fully discrete algorithm; then we provide details on the experiments, i.e., we specify concrete choices for the parameters, explain what solution statistics are collected as well as how the experiments' data are generated; we close the section by presenting the results and relating them to the research questions.
   
\subsection{Fully discrete algorithm}

\subsubsection{Finite Element Method}  
Details on the Finite Element Method can be found e.g. in the books~\cite{brezzi1,brenner1,Ern2021} and the references therein. We shortly introduce our notation. Let $\mathcal{D} \subset \mathbb{R}^2$ be a Lipschitz domain.

Let $\mathcal{T}_h$ denote a regular partition (triangulation)
of~$\mathcal{D}$ (no hanging nodes), which consists of closed $n$-simplices called \emph{elements}. For each element ($n$-simplex) $K\in \mathcal{T}_h$, we denote by $h_K$ the diameter of $K$. We define the maximal mesh-size by  $h := \max_{K \in \mathcal{T}_h} h_K$. 

For $r \in \setN_0$, let $\mathscr{P}_r(K)$ denote polynomials on $K$ of degree less than or equal to $r$; we define the vector-valued finite element spaces by
\begin{subequations} \label{def:velocity-FEM}
\begin{align}\label{def:Xh}
    X_h &:= \set{\bfv \in W^{1,\infty}(\mathcal{D}) \,:\, \bfv|_K \in (\mathscr{P}_r(K))^2 
      \,\,\forall K\in \mathcal{T}_h},\\ \label{def:Vh}
    V_h &:= X_h \cap W^{1,2}_{0}(\mathcal{D}).
\end{align} 
\end{subequations}
Similarly, we define the scalar-valued finite element spaces by
\begin{subequations} \label{def:pressure-FEM}
\begin{align}\label{def:Yh}
    Y_h &:= \set{q \in L^\infty(\mathcal{D}) \,:\, q|_K \in \mathscr{P}_r(K) \,\,\forall K\in \mathcal{T}_h},\\ \label{def:Qh}
   Q_h &:= Y_h \cap \left\{q \in L^\infty(\mathcal{D}) \,:\, \int_{\mathcal{D}} q \,\dd x  = 0\right\}.
\end{align} 
\end{subequations}
The vector-valued and scalar-valued finite element spaces are the approximate spaces for velocity and pressure, respectively. Additionally, the spaces~$V_h$ and $Q_h$ incorporate no-slip boundary conditions for the velocity and mean-value free condition for the pressure, respectively. The space of discretely divergence-free velocity is given by: 
\begin{align*}
 V_h^{\circ} := \left\{\bfv \in V_h\,: \, \forall q \in Q_h \,\int_{\mathcal{D}} \Div \bfv \,q \, \dd x = 0 \right\}.
\end{align*}

A mixed finite element pair $(V_h,Q_h)$ is called \textit{stable} if the so-called Ladyshenskaya--Babuska--Brezzi~\eqref{eq:LBB} condition is satisfied, i.e., there exists a constant~$\beta >0$ (independent of $h$) such that 
\begin{align} \tag{LBB} \label{eq:LBB}
\inf_{q \in Q_h\backslash \{0\}}\sup_{\bfv \in V_h\backslash \{0\} } \frac{\int_{\mathcal{D}}\Div \bfv \, q \, \dd x }{\norm{\nabla \bfv}_{L^2_x} \norm{q}_{L^2_x} } \geq \beta.
\end{align}
The~\eqref{eq:LBB} condition is a necessary condition for showing well-posedness of mixed formulations. Without it, it is impossible to uniquely reconstruct the pressure once the velocity has been constructed.

\subsubsection{Recovering the energy identity} We have seen that the analytic and semi-discrete solutions satisfy pathwise energy identities; see~\eqref{eq:en0} and~\eqref{lem:3.1a0}. However, these identities heavily rely on the incompressibility of the solutions~$\bfu$ and~$(\bfu_m)_{m=1}^M$, and the transport fields~$\bfsigma_k$, $k=1,\ldots,K$. Mixed finite elements generally fail to preserve the (pointwise) incompressibility constraint, which hinders an extension of the pathwise energy identities to fully discrete algorithms. We recover the energy identity by replacing the tri-linear term
\begin{align*}
B(\bfu,\bfv,\bfw) := \int_{\mathcal{D}}(\bfu \cdot \nabla )\bfv \cdot \bfw  \dx, \qquad \bfu \in L^\infty(\mathcal{D}), \qquad \bfv, \bfw \in W^{1,2}(\mathcal{D})
\end{align*}
with 
\begin{align} \label{eq:def-C}
C(\bfu,\bfv,\bfw) := \frac{1}{2} \int_{\mathcal{D}}(\bfu \cdot \nabla )\bfv \cdot \bfw  \dx  -  \frac{1}{2} \int_{\mathcal{D}} (\bfu \cdot \nabla) \bfw  \cdot  \bfv  \dx.
\end{align}
This approach was originally introduced by Temam in~\cite{Temam1968}. Most importantly, $C(\bfu, \cdot,\cdot)$ is anti-symmetric independently of the incompressibility of the transport field~$\bfu$. Additionally, the change has no effect for incompressible transport fields that vanish on the boundary; indeed, if $\Div \bf u = 0$ and $\bfu$ vanishes on the boundary, then integration by parts shows that~$B(\bfu,\bfv,\bfw) = C(\bfu,\bfv,\bfw)$  for all $\bfv, \bfw \in W^{1,2}(\mathcal{D})$.

\subsubsection{The algorithm} Let $(V_h,Q_h)$ be a stable mixed finite element pair. The fully discrete algorithm reads as follows: Iteratively construct the sequence of approximate velocity~$(\bfu_{m}^h)_{m=1}^M \subset V_h$ and pressure~$(\pi_{m}^h)_{m=1}^M \subset Q_h$ such that for all $\bfvarphi^h \in V_h$, $q^h \in Q_h$, $m = 0,\ldots, M-1$ and $\mathbb{P}$-a.s.
\begin{subequations}\label{algo:full-disc}
\begin{align}
&\begin{aligned} \label{algo:velo}
&\int_{\mathcal{D}}\bfu_{m+1}^h \cdot \bfvarphi^h \dx +\Delta t\bigg(\mu \int_{\mathcal{D}} \nabla\bfu_{m+1}^h:\nabla\bfphi^h \dx  - \int_{\mathcal{D}}\pi_{m+1}^h \, \Div \bfvarphi^h \dx  \bigg)\\
&\hspace{5em} + \Delta t \,C(\bfu_{m+1/2}^h, \bfu_{m+1/2}^h, \bfvarphi^h) \\
&\qquad =\int_{\mathcal{D}}\bfu_{m}^h \cdot\bfvarphi^h \dx+\sum_{k=1}^K C(\bfsigma_k , \bfu_{m+1/2}^h, \bfvarphi^h)\Delta_m W_k ,
\end{aligned}\\ \label{algo:pres}
&\int_{\mathcal{D}} \Div \bfu_{m+1/2}^h \, q^h \dx = 0,
\end{align}
\end{subequations}
where $\bfu_{m+1/2}^h = \tfrac{1}{2} (\bfu_{m+1}^h + \bfu_{m}^h)$. Well-posedness of this algorithm can be derived in the same way as for the semi-discretisation~\eqref{tdiscr}. Moreover, the approximate velocity satisfies the same pathwise energy identity as the solution to the semi-discretisation, which can be seen by choosing~$\bfvarphi^h = \bfu_{m+1/2}^h$ in~\eqref{algo:velo}. Indeed, this choice yields
\begin{align*} 
\tfrac{1}{2}\norm{\bfu_{m+1}^h}_{L^2_x}^2 + \tfrac{1}{4} \Delta t \mu \norm{\nabla \bfu_{m+1}^h}_{L^2_x}^2 + \Delta t\, \mu  \norm{ \nabla \bfu_{m+1/2}^h}_{L^2_x}^2  = \tfrac{1}{2}\norm{\bfu_{m}^h}_{L^2_x}^2  + \tfrac{1}{4} \Delta t \mu \norm{\nabla \bfu_{m}^h}_{L^2_x}^2,
\end{align*}
where we used the anti-symmetry of $C(\bfu,\cdot,\cdot)$ for $\bfu \in \{ \bfu_{m+1/2}^h, \bfsigma_1, \ldots, \bfsigma_K \}$. 

Neglecting the gradient contribution of the arithmetic mean, we observe that the 
\textit{enhanced} energy
\begin{align*}
\{0,\ldots,M\} \ni m \mapsto \tfrac{1}{2}\norm{\bfu_{m}^h}_{L^2_x}^2+ \tfrac{1}{4} \Delta t \mu \norm{\nabla \bfu_{m}^h}_{L^2_x}^2
\end{align*}
is monotonically decreasing. The enhanced energy consists of two terms: the kinetic energy and an algorithmically induced energy. Since the algorithmically induced energy is scaled by the time step size, it will become negligible eventually.


\subsection{Experiment setup}
The presentation of this section follows Section~7 of the last author's article~\cite{Le2024}.

\subsubsection{Model parameters}
We use the following fixed model parameters:
\begin{itemize}
\item two-dimensional unit square $\mathcal{D} = (0,1)^2$;
\item final time $T=1$;
\item viscosity $\mu = 1$;
\item single stochastic transport field $K = 1$.
\end{itemize}

\subsubsection{Spatial discretisation} \label{subsubsec:spatial-disc}
\begin{figure}
\includegraphics[width=0.5\textwidth]{\grSrc{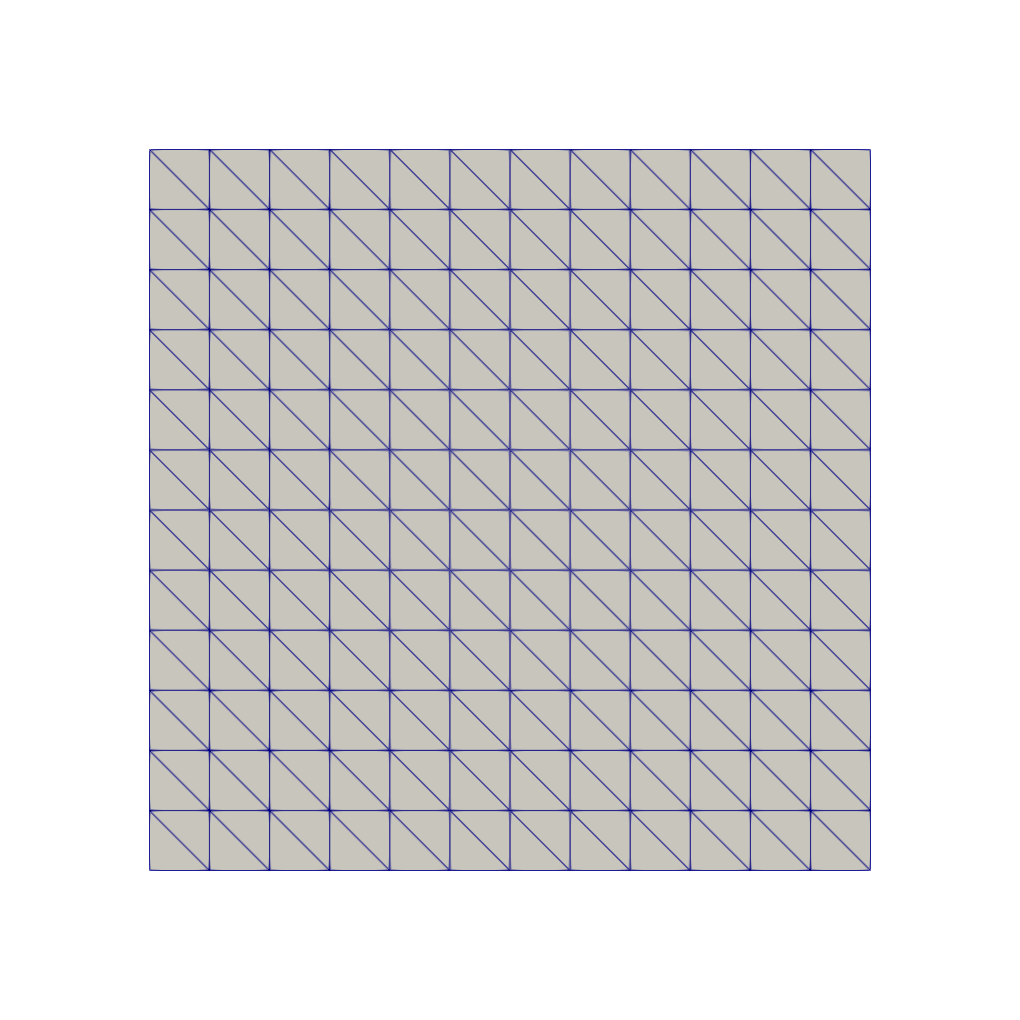}}
\caption{Triangulation of the unit square that is used in all experiments.}
\label{fig:Mesh}
\end{figure}

We discretise space with Taylor--Hood elements -- a stable mixed finite element pair; see e.g.~\cite{Taylor1973} -- generated by the triangulation presented in Figure~\ref{fig:Mesh}. The pair consists of continuous velocity and continuous pressure approximate spaces, which are given by~\eqref{def:Xh} and~\eqref{def:Yh} with $r = 2$ and $r = 1$, respectively. Taylor-Hood elements are merely discretely divergence free; thus, generally violating the incompressibility constraint.

\subsubsection{Sampling strategy}
We employ a Monte-Carlo approach for the discretisation of the probability space.

Let $M, L \in \mathbb{N}$ be the number of time-steps and samples, respectively. Realisations of the random vector $(\Delta_m W(\omega_\ell) )_{m=1}^{M}$, $\ell =1,\ldots,L $, are replaced by 
\begin{align} \label{eq:sample-random-inc}
Z_m^\ell \approx \Delta_m W(\omega_\ell), \qquad m =1,\ldots,M,\quad  \ell =1, \ldots, L,
\end{align}
which are generated independently by a pseudo-random number generator.

\subsubsection{Implemented algorithm}
We slightly modify the setting in~\eqref{algo:full-disc} by adding a deterministic force~$\bff$. Given an initial condition~$\bfu_0$, a noise coefficient~$\bfsigma$ and a deterministic forcing~$\bff$, we implement the fully discrete system of equations:
\begin{itemize}
\item (Initialise) For all $\ell =1,\ldots,L$, set $\bfu_0^\ell := \bfu_0$;
\item (Time-stepping) For all $m =0,\ldots,M-1$ and $\ell =1,\ldots,L$, define $\bfu_{m+1}^\ell \in V_h$ and $\pi_{m+1}^\ell \in Q_h$ by iteratively solving, for all $\bfphi^h \in V_h$ and $q^h \in Q_h$, 
\begin{subequations} \label{algo:Experiment-stepping}
\begin{align}
&\begin{aligned} 
&\int_{\mathcal{D}}\bfu_{m+1}^\ell \cdot \bfvarphi^h \dx +\Delta t\bigg(\mu \int_{\mathcal{D}} \nabla\bfu_{m+1}^\ell:\nabla\bfphi^h \dx  - \int_{\mathcal{D}}\pi_{m+1}^\ell \, \Div \bfvarphi^h \dx  \bigg)\\
&\hspace{5em} + \Delta t \,C(\bfu_{m+1/2}^\ell, \bfu_{m+1/2}^\ell, \bfvarphi^h) \\
&\qquad =\int_{\mathcal{D}}\bfu_{m}^\ell \cdot\bfvarphi^h \dx+ \Delta t\, \int_{\mathcal{D}} \bff \cdot \bfphi^h\, \dd x  +\Xi(\bfsigma,\bfu_{m+1/2}^\ell,\bfphi^h) Z_m^\ell ,
\end{aligned}\\ 
&\int_{\mathcal{D}} \Div \bfu_{m+1/2}^\ell \, q^h \dx = 0,
\end{align}
\end{subequations}
where $\bfu_{m+1/2}^\ell = \tfrac{1}{2}(\bfu_{m+1}^\ell+\bfu_{m}^\ell)$, and the noise coefficient is either
\begin{itemize}
\item (additive noise) $\Xi(\bfsigma,\bfu,\bfphi) = \int_{\mathcal{D}} \bfsigma \cdot \bfphi \, \dd x$,
\item (multiplicative noise) $\Xi(\bfsigma,\bfu,\bfphi) = \norm{\bfsigma}_{L^2_x} \int_{\mathcal{D}} \bfu \cdot \bfphi \, \dd x$,
\item (transport noise) $\Xi(\bfsigma,\bfu,\bfphi) = C(\bfsigma,\bfu,\bfphi)$ where  $C$ is given by~\eqref{eq:def-C} or 
\item (no noise) $\Xi(\bfsigma,\bfu,\bfphi) = \bf0$.
\end{itemize}
\end{itemize}
The proposed fully discrete algorithm uses a semi-implicit time stepping scheme,  which requires us to solve a quadratic non-linear system of equations in each iteration. While this increases the solving time, we gain more control on the generated velocity approximation due to the pathwise energy identity. Designing efficient solvers for non-linear system of equations is in itself an important research field. However, it isn't the scope of this article to discuss the algorithm's solve step; instead we assume the solve is done by an oracle. Later we will give details on the oracle used in our numerical simulations.

\subsubsection{Monitored statistics}
Motivated by Theorem~\ref{thm:3.1}, we exclusively trace the dependence of velocity and pressure on the time resolution either expressed in terms of the number of steps~$M$, or the step size~$\Delta t = T/M$. We monitor two statistics: kinetic energy evolution and time convergence.

We will trace the time evolution of the expected kinetic energy 
\begin{align*}
\{0,\ldots,M\} \ni m \mapsto \mathbb{E}\left[ \tfrac{1}{2} \norm{\bfu_m^h(M)}_{L^2_x}^2 \right].
\end{align*}
Here we use $(\bfu_m^h(M))_{m=1}^M$ to denote the solution of~\eqref{algo:full-disc}, where we emphasize the dependence of the solution on the time resolution.   

We measure time convergence of the algorithm with respect to the following velocity and pressure distances:
\begin{itemize}
\item $L^2\big(\Omega;L^\infty(0,T;L^2(\mathcal{D}))\big): \quad \mathbb{E}\left[ \max_{m\leq M} \norm{ \bfu (t_m^M) - \bfu_m^h(M)}_{L^2_x}^2 \right]$;
\item $L^2\big(\Omega;L^2(0,T;W^{1,2}_0(\mathcal{D}))\big): \quad \mathbb{E}\left[ \sum_{m=1}^M \Delta t\, \norm{ \nabla \bfu (t_m^M) - \nabla \bfu_m^h(M)}_{L^2_x}^2 \right]$;
\item $L^2\big(\Omega;L^2(0,T;L^2(\mathcal{D}))\big): \quad \mathbb{E}\left[ \sum_{m=1}^M \Delta t \, \norm{ \pi (t_m^M) - \pi_m^h(M)}_{L^2_x}^2 \right]$;
\item $L^2\big(\Omega;W^{-1,2}(0,T;L^2(\mathcal{D}))\big): \quad \mathbb{E}\left[ \sum_{m=1}^M \Delta t \, \Big\| \int_0^{t_m^M} \pi \dd s - \sum_{s=1}^m \Delta t\, \pi_s^h(M) \Big\|_{L^2_x}^2 \right]$.
\end{itemize}
The velocity distances are the natural extensions of the error measures introduced in Theorem~\ref{thm:3.1}. It should be noted, the time convergence in the pressure is still an open problem. The observation of the pressure distances provides us with numerical evidence for convergence.

Numerically, the kinetic energy as well as the velocity and pressure distances are inaccessible for two reasons: expectations can only be computed approximately and the analytic solution is unknown. We overcome these issues by using an empirical approximation of expectations and replacing the analytic solution by a more refined approximation.

\subsubsection{Lifting vectors to functions}
The dependence of the function spaces on the time resolution can be removed by first lifting vectors to functions and afterwards comparing functions instead of vectors.
 
Let $E$ be a vector space and $M \in \mathbb{N}$. Let $U=(u_m)_{m=0}^M$ with $u_m \in E$ for all $m \in \{0,\ldots,M\}$. We define the piecewise constant interpolation by 
\begin{align*}
\mathscr{L}[U](t) &:= u_M + \sum_{m=1}^M (u_{m-1} - u_M) \chi_{J_m^M}(t),
\end{align*}
where $\chi$ denotes the indicator function, $J_m^M = [t_{m-1}^M,t_m^M)$ and $t_m^M = T \, m/M$.

\subsubsection{Comparison on different time scales} Instead of studying the distance of our approximation at a fixed time resolution and the (unknown) analytic solution, we replace the analytic solution by a more refined approximation.

Let $M_c$, $M_f \in \mathbb{N}$ be 'coarse' and 'fine' time discretisations, respectively. Moreover, let $\bfU(M_c) = (\bfu_m^h(M_c))_{m=1}^{M_c}$, $ \bfU(M_f) = (\bfu_m^h(M_f))_{m=1}^{M_f}$ and $\Pi(M_c) = (\pi_{m}^h(M_c))_{m=1}^{M_c}$, $\Pi(M_f) = (\pi_m^h(M_f)_{m=1}^{M_f}$ be the corresponding 'coarse' and 'fine' velocity and pressure vectors, respectively. We define the following distances on vectors:
\begin{subequations} \label{def:choice-of-distance}
\begin{align}
\mathrm{d}_{L^\infty_t L^2_x}^{\mathrm{vel}}\big( \bfU(M_c), \bfU(M_f) \big) &:= \norm{\mathscr{L}[\bfU(M_c)] - \mathscr{L}[\bfU(M_f)] }_{L^\infty_t L^2_x}, \\
\mathrm{d}_{L^2_t W^{1,2}_{0,x}}^{\mathrm{vel}}\big( \bfU(M_c), \bfU(M_f) \big) &:= \norm{\nabla \mathscr{L}[\bfU(M_c)] - \nabla \mathscr{L}[\bfU(M_f)]  }_{L^2_t L^2_x}, \\
\mathrm{d}_{L^2_t L^2_x}^{\mathrm{pre}}\big( \Pi( M_c) , \Pi(M_f )\big) &:= \norm{\mathscr{L}[\Pi(M_c)] - \mathscr{L}[\Pi(M_f)] }_{L^2_t L^2_x}, \\
\mathrm{d}_{W^{-1,2}_t L^2_x}^{\mathrm{pre}}\big( \Pi( M_c) , \Pi(M_f )\big) &:= \bigg\|\int_0^{\cdot} \mathscr{L}[\Pi(M_c)](s) - \mathscr{L}[\Pi(M_f)](s) \,\dd s \bigg\|_{L^2_t L^2_x}.
\end{align}
\end{subequations}

\subsubsection{Empirical approximation of expectation}
We replace incomputable probabilistic mean-values by computable empirical approximations. 

Recall that $L$ denotes the sample size. We define the mean kinetic energy evolution by
\begin{align}\label{eq:statistic-energy}
\{0,\ldots,M\} \ni m \mapsto \frac{1}{L} \sum_{\ell=1}^L \tfrac{1}{2} \norm{\bfu_m^\ell(M)}_{L^2_x}^2.
\end{align}

For a given sample index $\ell \in \{1,\ldots,L\}$ and time resolution $M \in \{M_c, M_f\}$, let $\bfU^\ell(M)$ and $\Pi^\ell(M)$ denote approximate velocity and pressure vectors with resolution~$M$ generated by the $\ell$-th sample, respectively. We define empirical convergence measures by
\begin{subequations} \label{eq:statistical-mean}
\begin{align}
\mathrm{E}^{\mathrm{vel},L}_{L^\infty_t L^2_x}(M_c,M_f) &:= \left( \frac{1}{L} \sum_{\ell=1}^L \left( \mathrm{d}_{L^\infty_t L^2_x}^{\mathrm{vel}}\big( \bfU^\ell(M_c), \bfU^\ell(M_f) \big) \right)^2 \right)^{1/2}, \\
\mathrm{E}^{\mathrm{vel},L}_{L^2_t W^{1,2}_{0,x}}(M_c,M_f) &:= \left( \frac{1}{L} \sum_{\ell=1}^L \left( \mathrm{d}_{L^2_t W^{1,2}_{0,x}}^{\mathrm{vel}}\big( \bfU^\ell(M_c), \bfU^\ell(M_f) \big) \right)^2 \right)^{1/2}, \\
\mathrm{E}^{\mathrm{pre},L}_{L^2_t L^2_x}(M_c,M_f) &:= \left( \frac{1}{L} \sum_{\ell=1}^L \left( \mathrm{d}_{L^2_t L^2_x}^{\mathrm{pre}}\big( \Pi^\ell(M_c) , \Pi^\ell(M_f )\big) \right)^2 \right)^{1/2}, \\
\mathrm{E}^{\mathrm{pre},L}_{W^{-1,2}_t L^2_x}(M_c,M_f) &:= \left( \frac{1}{L} \sum_{\ell=1}^L \left( \mathrm{d}_{W^{-1,2}_t L^2_x}^{\mathrm{pre}}\big( \Pi^\ell(M_c) , \Pi^\ell(M_f )\big) \right)^2 \right)^{1/2}.
\end{align}
\end{subequations}

\subsubsection{Simultaneous sampling on different time scales}
We need to ensure that coarse and fine approximates are computed on the same pseudo-random event. Notice that the pseudo-random numbers defined in~\eqref{eq:sample-random-inc} depend on the time resolution. In particular, different time resolutions cannot be sampled independently, as can be seen on the analytic level: the random vectors 
\begin{align*}
(\Delta_m W)_{m=1}^{M_c} \hspace{3em} \text{ and }  \hspace{3em} (\Delta_m W)_{m=1}^{M_f} 
\end{align*}
are correlated, requiring us to modify the generation of the pseudo-random numbers to account for dependencies on different time scales. If $M_f = r M_c$ for some $r \in \mathbb{N}$, then we can write the coarse increments in terms of the fine ones, yielding a simple reconstruction rule:
\begin{align} \label{eq:sample-reconst}
Z_m^{\ell,\mathrm{c}} &= \sum_{n=1}^{r} Z_{(m-1)r +n}^{\ell,\mathrm{f}}, \qquad m =1,\ldots,M_c,\quad \ell=1,\ldots,L.
\end{align}
In other words, the fine pseudo-random numbers fully determine the values of the coarse ones.

\subsubsection{Varying parameters}
We run two sets of experiments: the time evolution of kinetic energy for different noise types and the algorithm's time convergence for transport noise. Before we can state the chosen parameters in each experiment, we need to introduce some vector fields and operators.

Let 
\begin{align*}
\mathrm{Poly}(x,y) &= \begin{pmatrix}
x^2(1-x)^2(2-6y+4y^2)y \\
-y^2(1-y)^2(2-6x+4x^2)x
\end{pmatrix}, \\
\mathrm{Poly_{noBC}}(x,y)& = \mathrm{Poly}(x,y) + \begin{pmatrix}
1\\
1
\end{pmatrix},\\
\mathrm{Poly_{noDiv}}(x,y)& = \mathrm{Poly}(x,y) + \begin{pmatrix}
x(1-x)y(1-y)\\
x(1-x)y(1-y)
\end{pmatrix}.
\end{align*}
The first vector field is smooth, divergence-free and vanishes on the boundary of~$\mathcal{D}$. The second and third vector fields violate boundary conditions and the incompressibility constraint, respectively. Moreover, let
\begin{align*}
\mathrm{Trig}(x,y) &=  \begin{pmatrix}
\sin(2 \pi x)\sin(4 \pi y) \\
-\sin(4\pi x) \sin(2 \pi y)
\end{pmatrix}.
\end{align*}
This vector field is smooth and vanishes on the boundary; but it isn't incompressible.

Additionally, we use the $L^2$-projection onto discretely divergence free velocity with vanishing trace given by
\begin{align*}
L^2_x \ni \bfw \mapsto \mathfrak{P}_h^\mathrm{vel} \bfw := \arg \min_{\bfv \in V_h^\circ} \norm{\bfv - \bfw}_{L^2_x}.
\end{align*}
We access this projection by implementing the discrete Helmholtz decomposition $\mathfrak{P}_h = (\mathfrak{P}_h^\mathrm{vel}, \mathfrak{P}_h^\mathrm{pre}): L^2(\mathcal{D}) \to V_h \times Q_h $ defined by
\begin{alignat*}{2}
&\forall \bfphi^h \in V_h: \quad \int_{\mathcal{D}} \mathfrak{P}_h^\mathrm{vel} \bfw \cdot \bfphi^h\, \dd x - \int_{\mathcal{D}} \mathfrak{P}_h^\mathrm{pre} \bfw \, \Div \bfphi^h \, \dd x &&= \int_{\mathcal{D}} \bfw \cdot \bfphi^h \,\dd x, \\
&\forall q^h \in Q_h: \quad  \int_{\mathcal{D}} \Div  \mathfrak{P}_h^\mathrm{vel} \bfw \, q^h \, \dd x &&= 0.
\end{alignat*}
The first component of this vector-valued operator is the aforementioned $L^2$-projection; the second component is the projection onto discrete gradients.

\textbf{SOE--1: Experimental time evolution of kinetic energy for different noise types.} The first set of experiments investigate the influence of different noise types on the kinetic energy. 

\begin{table}
\begin{tabular}{|c|c|}
\toprule
Noise type & Noise coefficient~$\Xi(\bfsigma,\bfu,\bfphi)$ \\
 \midrule
 additive & $\int_{\mathcal{D}} \bfsigma \cdot \bfphi \, \dd x$  \\
 multiplicative & $ \norm{\bfsigma}_{L^2_x} \int_{\mathcal{D}} \bfu \cdot \bfphi \, \dd x$  \\
 transport & $\tfrac{1}{2} \int_{\mathcal{D}} (\bfsigma \cdot \nabla) \bfu \cdot \bfphi \, \dd x  - \tfrac{1}{2} \int_{\mathcal{D}} (\bfsigma \cdot \nabla)\bfphi  \cdot \bfu \, \dd x $  \\
 no & $\bf0$ \\
 \bottomrule
\end{tabular}
\vspace*{-0.9em}
\caption{Parameter configuration in SOE--1: Choice of noise coefficient for each noise type.}
\label{tab:experiments-energy}
\end{table}

We conduct four experiments by switching between additive, multiplicative, transport or no noise. In Table~\ref{tab:experiments-energy} we present noise coefficients used for each noise type. In all cases, we choose the same noise datum, initial condition and deterministic forcing. We project them to the space of discrete velocities by calling the \textit{firedrake.project} function, which we treat as a blackbox. The noise datum equals the vector field~$\mathrm{Poly}$ scaled by the factor~$10^3$; the initial condition is set to be the $L^2$-projection of the vector field~$\mathrm{Poly}$ onto the space of discretely divergence-free velocity. We additionally scale the initial velocity by a factor of~$10^3$; the deterministic forcing is chosen as the vector field~$\mathrm{Trig}$ scaled by the factor~$10^2$. In summary, we choose the following parameters:
\begin{align*}
\bfsigma = 10^3\, \mathrm{Poly}, \qquad \bfu_0 = 10^3 \, \mathfrak{P}_h^\mathrm{vel} \,\mathrm{Poly}, \qquad  \bff = 10^2\, \mathrm{Trig}.
\end{align*}

\textbf{SOE--2: Experimental time convergence for transport noise.} The second set of experiments test the algorithm's time convergence for different choices of initial conditions and noise data. 

\begin{table}
\begin{tabular}{|c|c|c|}
\toprule
Experiment & Initial condition~$\bfu_0$ & Noise datum~$\bfsigma$\\
 \midrule
 1 & $ \mathrm{Poly}$ & $ \mathrm{Poly}$ \\
 2 & $ \mathrm{Poly}$ & $ \mathrm{Poly_{noDiv}}$ \\
 3 & $ \mathrm{Poly}$ & $ \mathrm{Poly_{noBC}}$ \\
 4 & $ \mathrm{Poly_{noDiv}}$ & $ \mathrm{Poly}$ \\
 5 & $ \mathrm{Poly_{noBC}}$ & $ \mathrm{Poly}$ \\
 \bottomrule
\end{tabular}
\vspace*{-0.9em}
\caption{Parameter configuration in SOE--2: Choice of initial velocity and noise datum in each experiment.}
\label{tab:experiments}
\end{table} 
In Table~\ref{tab:experiments} we state our choice of initial velocity and noise datum in each experiment. We project initial velocity and noise datum to the space of discrete velocities by calling the \textit{firedrake.project} function. In Experiments~1, 2 and 3 we additionally invoke the $L^2$-projection to project the initial condition onto the space of discretely divergence-free velocity. Finally, the noise datum is scaled by a factor of $10^3$ in all experiments and no deterministic forcing $(\bff = 0$) is present.
 
\subsubsection{Summary}
We are ready to present the algorithm used for the sample generation of approximate velocity and pressure, the computation of the (mean) kinetic energy, and the time convergence test for velocity and pressure. The implementation is available at \href{https://github.com/joernwichmann/NSE-Transport}{https://github.com/joernwichmann/NSE-Transport}.

The spatial discretisation is fixed in all experiments (see Section~\ref{subsubsec:spatial-disc}). The first set of experiments~SOE--1 uses $M= 2^9$ time steps and a sample size of $L = 10^5$. The second set of experiments~SOE--2 uses a fine time resolution of $M_f =  2^9$ time steps, coarse time resolutions of $M_c \in \{2^2, 2^3, \ldots, 2^8 \}$ time steps and a sample size of $L = 10^4$. 
\begin{framed}
\begin{algorithm}[H]
\KwData{samples, time scales, space discretisation, varying parameters}
\KwResult{empirical study of kinetic energy evolution and time convergence}
\For{sample in samples}{
generate pseudo-random vector on finest time scale by~\eqref{eq:sample-random-inc}\;
\For{time scale in time scales}{
reconstruct coarse pseudo-random vector from finest one using~\eqref{eq:sample-reconst}\;
initialise time-stepping relative to time scale\;
\For{time step in time steps}{
solve~\eqref{algo:Experiment-stepping}\;
}
}
compute sample kinetic energy\;
compute sample distance of coarse and fine approximates\;
}
compute empirical mean of kinetic energy by~\eqref{eq:statistic-energy}\;
compute empirical mean of distances by~\eqref{eq:statistical-mean}\;
\caption{Pseudo-code used for our numerical simulations.}
\label{algo:Pseudo-code}
\end{algorithm}
\end{framed}
For actually solving~\eqref{algo:Experiment-stepping}, we rely on the \textit{firedrake.solve} function that internally uses the nonlinear solver~\textit{PETSc.SNES} with $10^{-8}$ as absolute and relative tolerance parameters, and the parallel sparse direct solver~\textit{MUMPS}.

\subsection{Results}

\subsubsection{SOE--1: Kinetic energy} \label{sec:mean-kinetic-energy}
The influence of different types of noise (additive, multiplicative, transport and no-noise) on the kinetic energy are shown in Figures~\ref{fig:kinetic},~\ref{fig:stationary-all} and~\ref{fig:energy}: Figure~\ref{fig:kinetic} displays the time evolution of individual energy trajectories as well as their statistics (empirical mean and standard deviation); Figure~\ref{fig:stationary-all} shows an empirical approximation of the stationary distributions; and Figure~\ref{fig:energy} illustrates the effect of each noise structure on the kinetic energy evolution and its stationary distribution individually. We summarize our observations:

\begin{figure}
\begin{subfigure}[t]{0.99\textwidth} 
\includegraphics[scale=0.425]{\grSrc{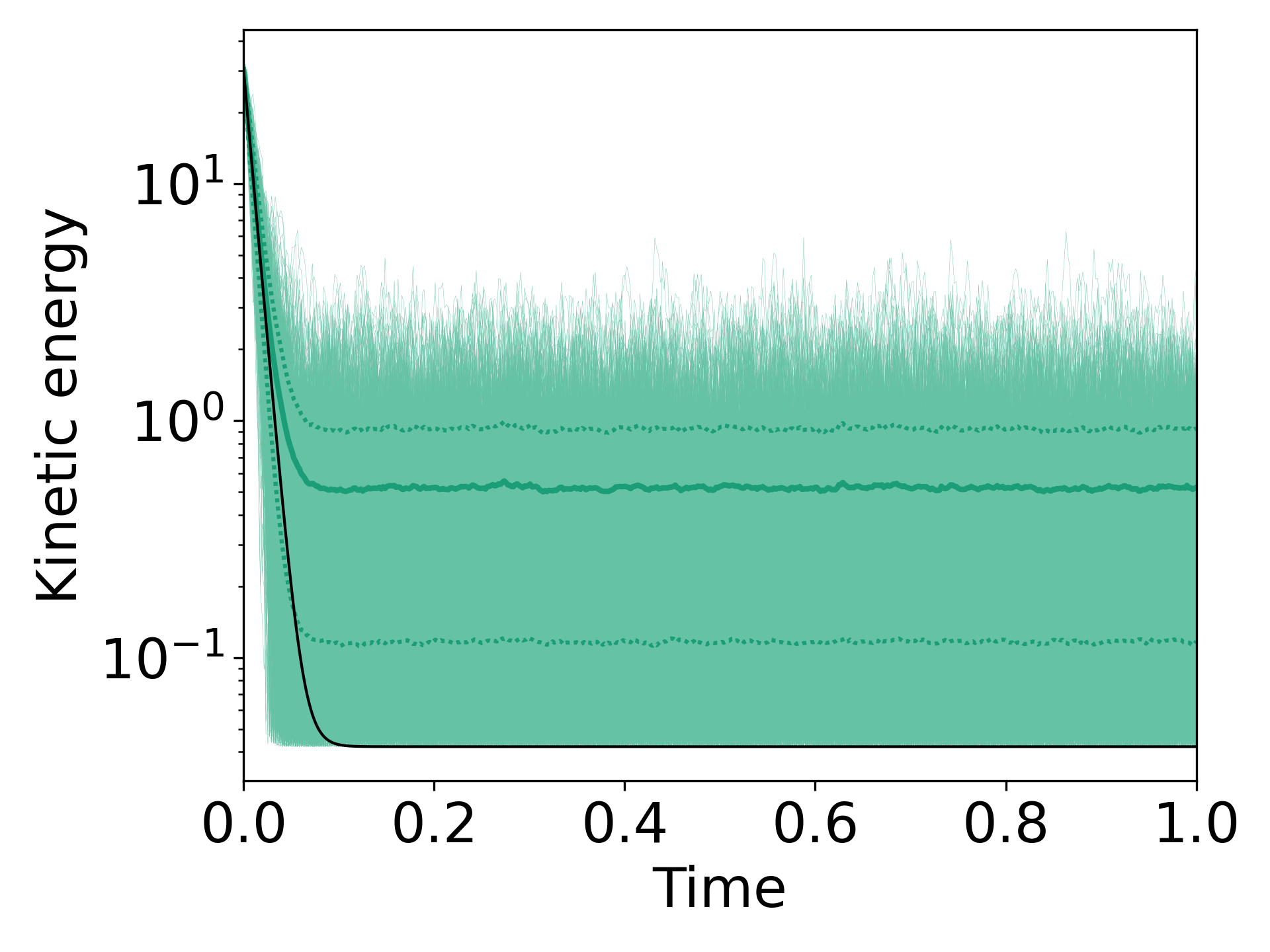}}%
\includegraphics[scale=0.425]{\grSrc{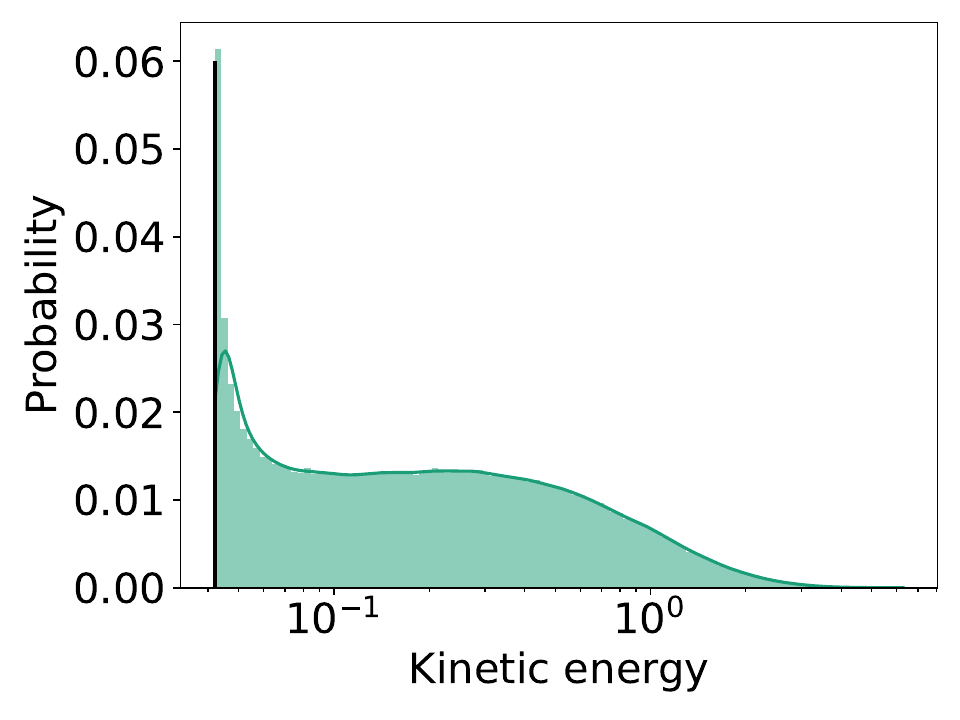}}
        \caption{Additive noise}
\end{subfigure}
\begin{subfigure}[t]{0.99\textwidth} 
\includegraphics[scale=0.425]{\grSrc{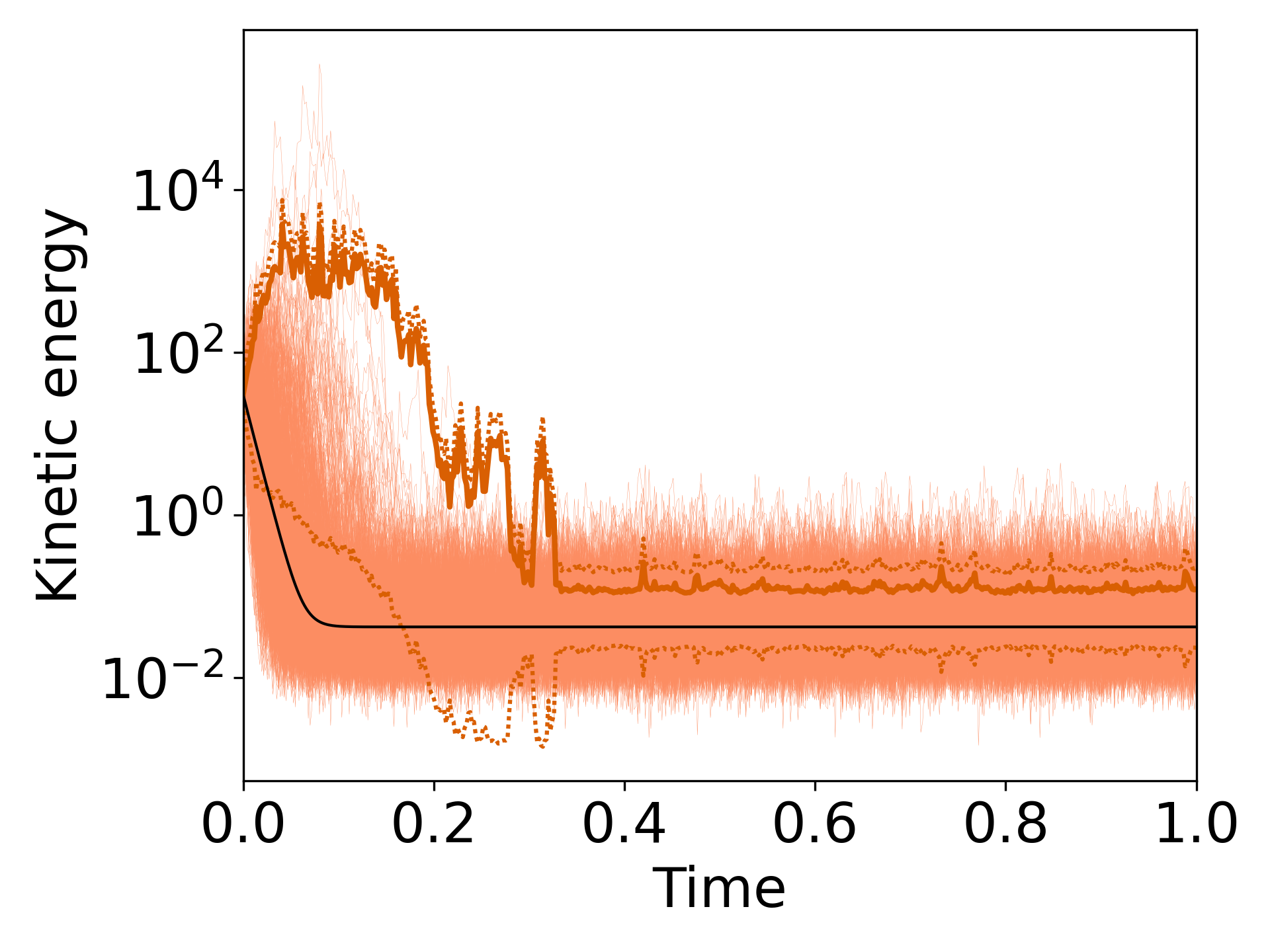}}%
\includegraphics[scale=0.425]{\grSrc{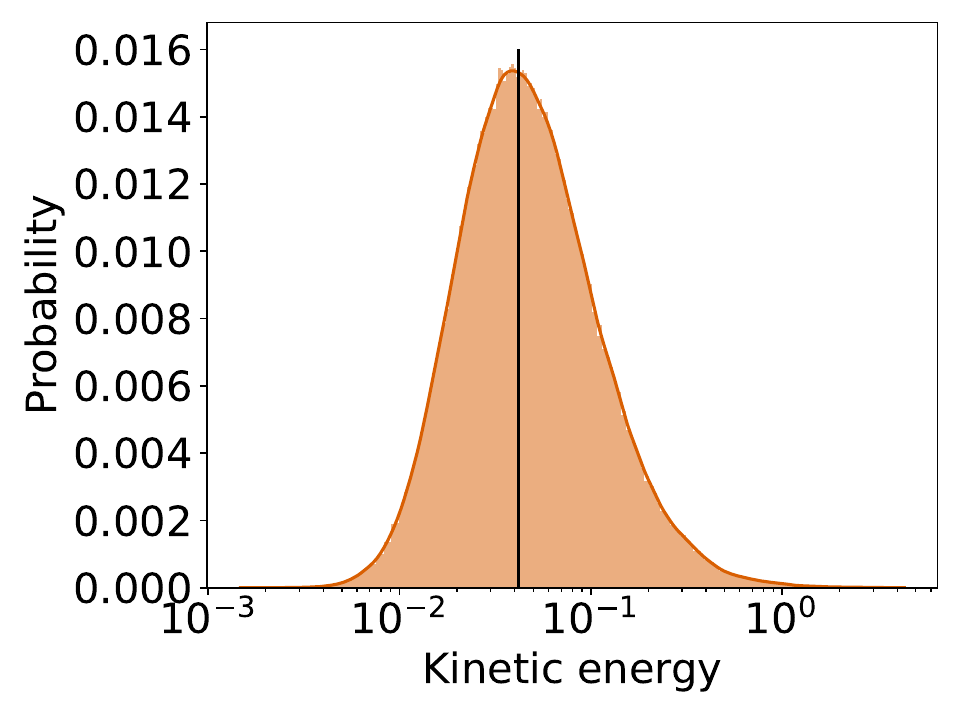}}
        \caption{Multiplicative noise}
\end{subfigure}
\begin{subfigure}[t]{0.99\textwidth} 
\includegraphics[scale=0.425]{\grSrc{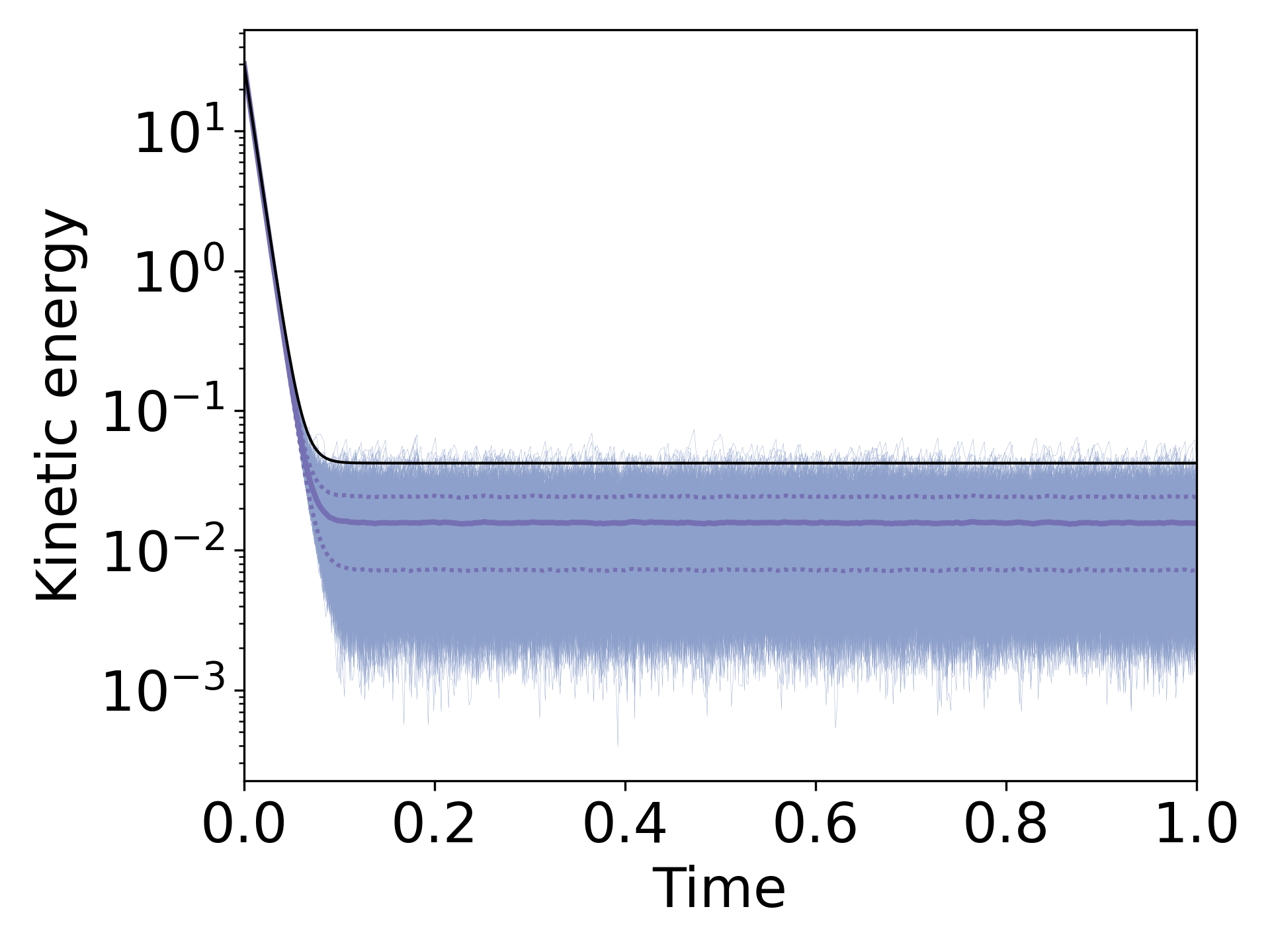}}%
\includegraphics[scale=0.425]{\grSrc{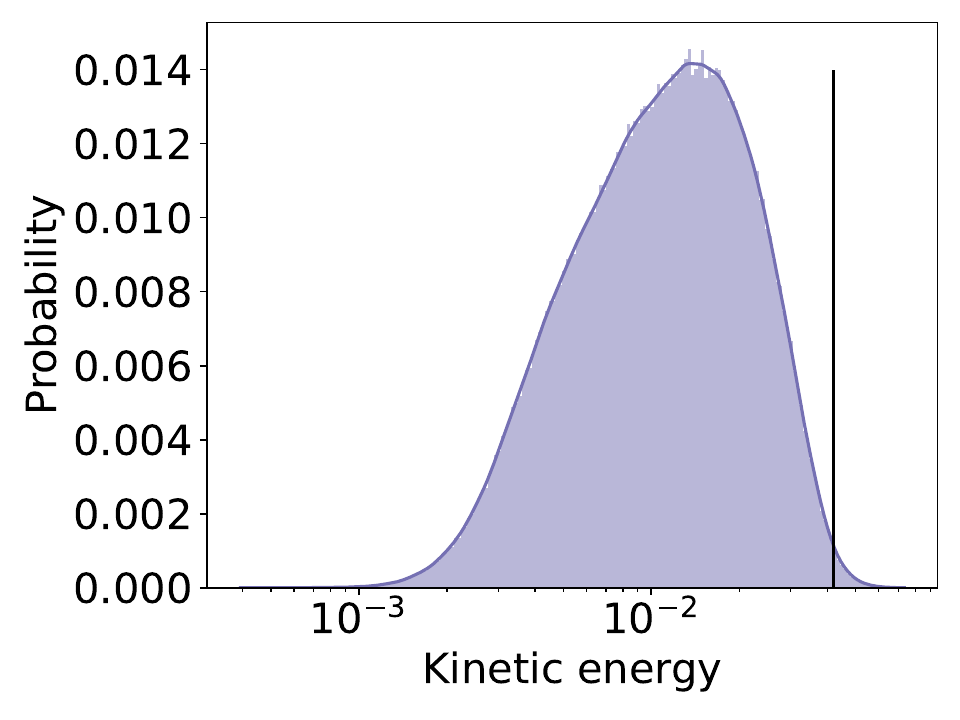}}
        \caption{Transport noise}
\end{subfigure}
\caption{Graphical illustration of the results of SOE--1: (left) time evolution of the kinetic energy for different noises. The deterministic kinetic energy evolution is shown in black. Thick lines and dotted lines show the mean energy and the mean energy plus or minus one standard deviation, respectively. The first 1,000 (out of 10,000) energy trajectories are shown in pale colours; (right) empirical approximation (based on 1,000 trajectories) of the stationary distributions of the kinetic energy for different noises. The deterministic stationary energy level is indicated by a black vertical line.}
 \label{fig:energy}
\end{figure} 

The deterministic system monotonically dissipates its kinetic energy until it reaches its stationary state, indicated by the constant energy level of around~$0.042$ at time~$0.13$. 

Transport noise leads to a lower mean kinetic energy compared to the deterministic energy: it is scaled by a factor of approximately~$1/3$. Initially, dissipation is dominant for transport noise as can be seen by the absence of fluctuations and small standard deviations. At the same time as the deterministic solution reaches its stationary state, randomness becomes more dominant (the standard deviation is approximately~$0.008$). Overall, the energy of trajectories rarely overcomes the deterministic stationary energy level. 

Multiplicative noise leads to a higher mean energy level compared to the deterministic one: it is scaled by a factor of approximately~$3$. We observe an initial mean energy peak, which is caused by rare trajectories with high energy levels. Large standard deviations indicate dominant stochastic behaviour for short times. After some time, the energy is dissipated and eventually reaches its stationary level at time~$0.32$. Its stationary distribution is almost symmetrically distributed (with respect to a logarithmic scale) around the deterministic stationary energy level. Even though reaching the stationary state, many fluctuations in the mean energy are still visible, which indicate sensitivity of the empirical approximation with respect to rare events. Its standard deviation is approximately~$0.1$.

Additive noise leads to a substantially higher mean energy compared to the no-noise case: it is scaled by a factor of approximately~$12$. Right from the start, dissipation dominants the energy dynamics. While the deterministic solution reaches its stationary state, randomness becomes more influential (the standard deviation is approximately~$0.4$). The density of the stationary distribution attains its maximum at the deterministic stationary energy level, which seems to be a lower bound for each trajectory's energy level. Even if they reach the energy level at an earlier time, they cannot surpass it. 

\subsubsection{SOE--2: Time convergence}
The results of the time convergence experiments are presented in Figure~\ref{fig:time-convergence}: it displays the time convergence of velocity, gradient of velocity, pressure, and time-integrated pressure. We observed the following: 

Velocity convergence on $L^2\big(\Omega;L^\infty(0,T; L^2(\mathcal{D}))\big)$ fails in all experiments but Experiment~1. In this case, interestingly, we also observe an accelerated convergence of rate~1 instead of the theoretically predicted rate~$0.5$. Experiments~1, 2 and~5 show velocity convergence on $L^2\big(\Omega; L^2(0,T; W^{1,2}_{0}(\mathcal{D}))\big)$ with rate~$0.6$, $0.2$ and~$0.33$, respectively. Initially, Experiment~1 reports velocity convergence with rate~$0.8$, which eventually reduces to~$0.4$. No convergence is detected in Experiments~3 and~4.

In all experiments, the pressure doesn't converge on $L^2\big(\Omega; L^2(0,T; L^2(\mathcal{D}))\big)$. Its convergence is restored on the weaker norm~$L^2\big(\Omega; W^{-1,2}(0,T; L^2(\mathcal{D}))\big)$ in all experiments but Experiment~4. Experiments~1, 2 and~5 report convergence with rate~$0.5$. Experiment~3 detects convergence with rate~$0.33$. 

\begin{figure}
\begin{subfigure}[t]{1.00\textwidth} 
\includegraphics[scale=0.9]{\grSrc{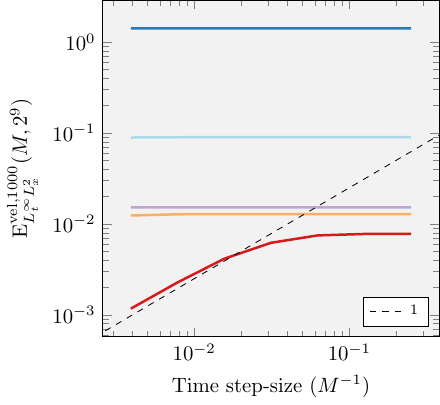}}%
\includegraphics[scale=0.9]{\grSrc{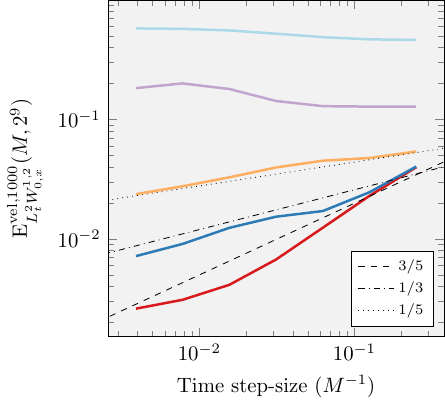}}
        \caption{Time convergence of velocity measured on $L^2_\omega L^\infty_t L^2_x$ (left) and $L^2_\omega L^2_t W^{1,2}_{0,x}$ (right).}
\end{subfigure}
\begin{subfigure}[t]{1.00\textwidth} 
\includegraphics[scale=0.9]{\grSrc{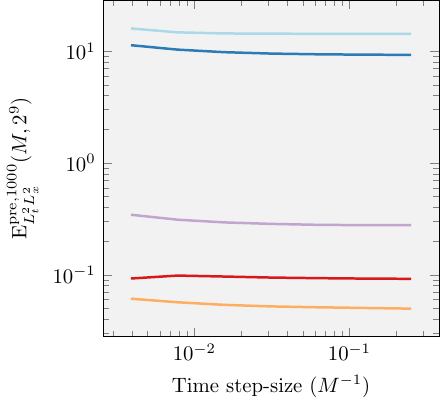}}%
\includegraphics[scale=0.9]{\grSrc{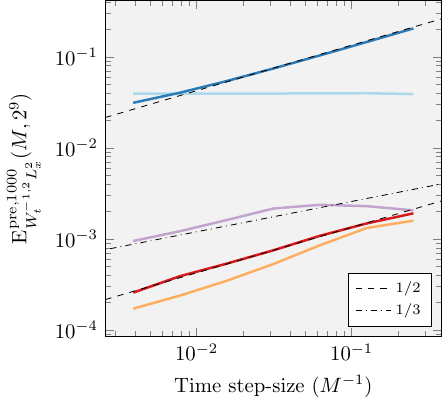}}
        \caption{Time convergence of pressure measured on $L^2_\omega L^2_t L^2_x$ (left) and $L^2_\omega W^{-1,2}_t L^2_x$ (right).}
\end{subfigure}
\caption{Graphical illustration of the results of SOE--2: time convergence of velocity and pressure. Colour encodes the experiments (Exp.~1:
{\protect\tikz \protect\draw[color=plot_color1, line width=1.2] (0,0) -- (0.5,0);},
Exp.~2:
{\protect\tikz \protect\draw[color=plot_color2, line width=1.2] (0,0) -- (0.5,0);},
Exp.~3:
{\protect\tikz \protect\draw[color=plot_color3, line width=1.2] (0,0) -- (0.5,0);},
Exp.~4:
{\protect\tikz \protect\draw[color=plot_color4, line width=1.2] (0,0) -- (0.5,0);},
Exp.~5:
{\protect\tikz \protect\draw[color=plot_color5, line width=1.2] (0,0) -- (0.5,0);}). The discontinuous lines are reference lines where the individual slopes are given in each figures' legend. 
 } 
 \label{fig:time-convergence}
\end{figure}

\subsection{Conclusions}
Different noises lead to different dynamics of the kinetic energy and different stationary distributions. While the dynamics in cases of transport and additive noise are controlled at all times by the energy dissipation, the multiplicative noise seems to be largely influenced by random effects and rare events; especially at short times. The densities of the stationary distributions show that additive noise exclusively leads to high energy levels (compared to the deterministic stationary energy level); multiplicative noise leads to fairly centred energy levels with occasional exceptions; and the energy of transport noise is mostly concentrated below the deterministic one. Depending on the modelling of the given situation it seems physically reasonable to choose a noise which does not affect the energy balance.
\footnote{In some complex situations such as weather forecast it might also be beneficial to consider a combination of additive and transport noise.} Since the energy is \emph{below} the deterministic one in the case of transport noise
there must be more dissipation. However, we currently do not understand the mechanism behind this. 

In the beginning of this section we asked: (a) can the torus be replaced by a Lipschitz domain, (b) can periodic boundary conditions be replaced by no-slip boundary conditions and (c) can constant transport fields be replaced by non-constant ones? Our numerical experiments provide an affirmative answer to these questions as long as the data (initial condition and noise datum) comply with the constraints, which are incompressibility and vanishing trace. In fact, our simulations even show that these data conditions are sharp: if initial velocity and noise datum satisfy the constraints, then velocity approximations converge with the theoretically predicted rate of Theorem~\ref{thm:3.1}.

Conversely, if either initial velocity or noise datum does not satisfy the constraints, then this has a decreasing effect on the convergence rate or it even destroys the convergence. 

Interestingly, our simulations don't provide numerical evidence for the necessity of Assumption~\eqref{ass:vis-dom-noise}. The transport field and viscosity, which we used in the time convergence experiments, violate this assumption by several orders of magnitude ($\Gamma_\bfsigma \approx 10^6 \gg 1 = \mu$). Still, we observe convergence, numerically. Therefore, we ask: is it possible to rigorously establish the convergence of our proposed algorithm without assuming~\eqref{ass:vis-dom-noise}?

In some experiments we observe convergence of pressure on $L^2\big(\Omega; W^{-1,2}(0,T; L^2(\mathcal{D}))\big)$. A theoretical explanation of this convergence has not been derived yet and is an interesting topic for future research. Only recently, related results for fully discrete algorithms for the stochastic Stokes system (no convection) forced by multiplicative noise have been derived by Feng, Vo and the third author~\cite{MR4286261}. In this case, they established an error estimate of pressure on $L^2\big(\Omega; W^{-1,2}(0,T; L^2(\mathcal{D}))\big)$ with rate~$0.5$, which coincides with our experimentally observed convergence rate. Whether their methods can also be applied for the Navier--Stokes equations remains to be investigated. From an analytical point of view, the regularity of the pressure is the most important prerequisite for the occurrence of pressure convergence. In~\cite{Wichmann2024}, the last author proved for generalized Stokes systems (the classical Stokes system is included) with multiplicative noise that pressure is more regular than $L^2\big(\Omega; W^{-1,2}(0,T; L^2(\mathcal{D}))\big)$: in fact, pressure almost belongs to the space $L^2\big(\Omega; W^{-1/2,2}(0,T; L^2(\mathcal{D}))\big)$. The differentiability gap between the spaces, in which the convergence is measured and the regularity is available, can be made arbitrarily close to~$0.5$. For specific algorithms, this gap can be used to verify the convergence of pressure, ultimately leading to a theoretically justified pressure convergence with rate~$0.5$ as done in~\cite{MR4286261}. In this article, we experimentally observe pressure convergence for almost all experiments in SOE--2. Therefore, it seems natural to ask: what are sufficient conditions for the initial condition and the noise coefficient such that pressure is regular?

\end{document}